\documentclass[12pt]{amsart}
\counterwithin{equation}{section}

\allowdisplaybreaks

\usepackage{amsmath,amssymb,amsbsy,amsthm,amscd}
\usepackage[all]{xy}
\usepackage{tikz}
\usepackage{mathrsfs}
\usepackage{graphicx}
\usepackage{enumerate}
	\definecolor{linkred}{rgb}{0.7,0.2,0.2}
	\definecolor{linkblue}{rgb}{0,0.2,0.6}
	\definecolor{linkgreen}{rgb}{0,0.6,0.2}
\usepackage[colorlinks,plainpages,
    linkcolor=linkblue,
    citecolor=linkgreen,
    urlcolor=linkred]{hyperref}

\newtheorem{theorem}{Theorem}[section]
\newtheorem{definition}[theorem]{Definition}
\newtheorem{lemma}[theorem]{Lemma}

\newtheorem{prop}[theorem]{Proposition}

\newtheorem{example}[theorem]{Example}
\newtheorem{remark}[theorem]{Remark}

\newcommand{\Ui}{\mathbf{U}^{\imath}}
\newcommand{\Sym}{{\mathrm{Sym}}}
\newcommand{\GL}{{\mathrm{GL}}}
\newcommand{\SO}{{\mathrm{SO}}}
\newcommand{\Sp}{{\mathrm{Sp}}}

\newcommand{\calD}{{\mathcal{D}}}
\newcommand{\calK}{{\mathcal{K}}}
\newcommand{\calO}{{\mathcal{O}}}
\newcommand{\calR}{{\mathcal{R}}}
\newcommand{\bfN}{{\mathbf{N}}}
\newcommand{\bGr}{\overline{\Gr}}
\newcommand{\pt}{\mathsf{pt}}

\DeclareMathOperator{\Hom}{Hom}
\DeclareMathOperator{\Gr}{Gr}
\DeclareMathOperator{\Eu}{Eu}
\DeclareMathOperator{\Res}{Res}
\DeclareMathOperator{\loc}{loc}

\begin{document}

\title[$\mathrm{i}$Quantum groups and Coulomb Branches]{Quivers with Involutions and Shifted affine $\mathrm{i}$Quantum groups via Coulomb Branches}

\author{Yaolong Shen}
\address[Yaolong Shen]{School of Mathematical Sciences, Key Laboratory of MEA (Ministry of Education)
\& Shanghai Key Laboratory of PMMP, East China Normal University, Shanghai 200241,
China}
\email{ylshen@math.ecnu.edu.cn}

\author{Changjian Su}
\address[Changjian Su]{Yau Mathematical Sciences Center, Tsinghua University, Beijing, China}
\email{changjiansu@mail.tsinghua.edu.cn}

\author{Zichang Wang}
\address[Zichang Wang]{Qiuzhen College, Tsinghua University, Beijing, China, 100084}
\email{w-zc21@mails.tsinghua.edu.cn}

\author{Rui Xiong}
\address[Rui Xiong]{Department of Mathematics, University of Michigan, Ann Arbor, MI 48109}
\email{xiongrui@umich.edu}


\begin{abstract}
We give a formula for quasi-minuscule monopole operators in $K$-theoretic quantized Coulomb branch algebras of cotangent type. We then study signed quivers with involutions, with orthogonal or symplectic gauge factors at the involution-fixed vertices. We construct a GKLO-type homomorphism from the corresponding shifted affine $\mathrm{i}$quantum group to an algebra of $q$-difference operators and prove that it factors through the quantized Coulomb branch algebra. These homomorphisms give geometric realizations of the shifted affine $\mathrm{i}$quantum groups in terms of dressed minuscule and quasi-minuscule monopole operators.
\end{abstract}

\maketitle
%\tableofcontents

\section{Introduction}

Quantized Coulomb branch algebras provide a geometric interpretation of difference-operator representations of quantum algebras. For a complex reductive group $G$ and a representation $\bfN$, the Braverman--Finkelberg--Nakajima construction equips the equivariant $K$-homology of the space $\calR_{G,\bfN}$ with a convolution product \cite{NakCou1,BFN18}. Loop rotation gives the quantization parameter $q$, and equivariant localization embeds the resulting algebra $\mathcal A_q(G,\bfN)$ into an algebra of $q$-difference operators. Under this embedding, monopole classes have explicit formulas in terms of the weights of the gauge and matter representations. For ordinary quiver gauge theories, these operators give geometric realizations of shifted Yangians/quantum affine algebras \cite{BFN19,FT19}. The purpose of this paper is to develop the corresponding constructions for shifted affine $\mathrm{i}$quantum groups arising from quivers with involutions.

A Coulomb branch realization requires two ingredients. One must find difference operators satisfying the defining relations of the quantum algebra, and one must identify those operators with classes in equivariant $K$-homology. The latter step imposes a geometric condition on formulas that are initially defined in a localization. For the quivers considered here, dressed minuscule monopole classes account for the generators at general linear and orthogonal gauge factors. A symplectic gauge factor requires a quasi-minuscule cocharacter. The corresponding affine Grassmannian orbit closure has an additional stratum, whose contribution must be included in the operator. This is the geometric issue addressed by our quasi-minuscule formula.

We treat signed quivers with involution. A nonfixed orbit gives a general linear gauge factor, while a fixed vertex gives an orthogonal or symplectic factor. A single current construction produces a GKLO-type homomorphism and its factorization through the associated $K$-theoretic quantized Coulomb branch algebra. 

\medskip
\noindent\textbf{Quasi-minuscule monopole operators.}
Let $\lambda$ be a dominant quasi-minuscule cocharacter. Its affine Grassmannian Schubert variety has the decomposition
\[
\bGr_\lambda=\Gr_\lambda\sqcup\{e\}.
\]
The open orbit $\Gr_\lambda$ is a line bundle $\mathcal L$ over $G/P_\lambda$, and the compactification $\mathbb P(\calO\oplus\mathcal L)$ resolves $\bGr_\lambda$. Extending the BFN space over this resolution gives a proper map
\[
\tilde p:\tilde{\calR}_{G,\bfN,\leq\lambda}\longrightarrow\calR_{G,\bfN,\leq\lambda}.
\]
We compute the localization of the pushforward of a dressed structure sheaf under $\tilde p$. Theorem~\ref{thm:quasi} shows, in particular, that for $f\in\mathbb C[T]^{W_\lambda}$,
\[
\sum_{w\in W^\lambda}
w\!\left(f\,\frac{\Eu\bigl(z^\lambda\bfN_\calO/(z^\lambda\bfN_\calO\cap\bfN_\calO)\bigr)}
{\Eu(T_\lambda\Gr_\lambda)}\right)(D_{w\lambda}-1)
\in\mathcal A_q(G,\bfN).
\]
Here $\Eu$ is the $K$-theoretic Euler class, and the inclusion is understood through the localization embedding. More precisely, the displayed operator agrees with the image of
$\tilde p_*\bigl(f[\calO_{\tilde{\calR}_{G,\bfN,\leq\lambda}}]\bigr)$ up to an element of $K^{G\times\widetilde{\mathbb C}_q^\times}(\pt)$. This is the $K$-theoretic counterpart of \cite[Theorem~2.4]{SSX26b}. It applies to arbitrary cotangent-type matter representations and supplies the monopole operator needed at the symplectic fixed vertex.

\medskip
\noindent\textbf{Quivers with involutions.}
Let $(Q,\tau)$ be a simply-laced quiver with an involution reversing arrows, with the signed data of Section~\ref{sec:quiver}. If $\tau$ has no fixed vertices, the gauge group $G_V^\sigma$ is a product of general linear groups, with one factor for each orbit of vertices. The matter representation is obtained from the involution-fixed arrow spaces together with the framing spaces. In particular, an arrow fixed by $\tau$ contributes a second exterior power in the convention used here. The geometric generators in this case are dressed minuscule monopole operators.

The fixed vertices $Q_0^\tau$ are partitioned into orthogonal ones $Q_O$ and the symplectic ones $Q_S$. For an orthogonal vertex $i$, the gauge group is $\SO(V_i)$ and the framing is a symplectic vector space $W_i$. For a symplectic vertex $i$, the gauge group is $\Sp(V_i)$ and the framing is an orthogonal even dimensional vector space $W_i$. In both cases, choose a maximal isotropic subspace $U_i\subset W_i$ so that $W_i=U_i\oplus U_i^*$.

By assumption, the fixed vertices are pairwise nonadjacent. Choose a
decomposition
\[
 Q_0=Q_0^+\sqcup Q_0^\tau\sqcup\tau(Q_0^+),
 \qquad Q_0^\tau=Q_O\sqcup Q_S.
\]
Choose a subset $Q_1^+\subset Q_1$ of representatives of the two-element arrow orbits.
Write $Q_1^\tau$ for the fixed arrows. With these data, we study the K-theoretic quantized Coulomb branch
\[
 \mathcal A_q(G\times F,\bfN)=K^{(G_\calO\times F_\calO)\rtimes
 \widetilde{\mathbb C}_q^\times}(\calR_{G,\bfN})
\]
associated to the following data:
\begin{equation*}
 G=\prod_{i\in Q_0^+}\GL(V_i)\times
   \prod_{i\in Q_O}\SO(V_i)\times
   \prod_{i\in Q_S}\Sp(V_i),
\end{equation*}
\begin{equation*}
 \bfN=\bigoplus_{h\in Q_1^+}\Hom(V_{s(h)},V_{t(h)})
 \oplus\bigoplus_{h\in Q_1^\tau}\Lambda^2 V_{t(h)}\oplus\bigoplus_{i\notin Q_0^\tau}\Hom(V_i,W_i)
 \oplus\bigoplus_{i\in Q_0^\tau}U_i\otimes V_i,
\end{equation*}
\[F=\prod_{i\notin Q_0^\tau}\GL(W_i)
   \times\prod_{i\in Q_0^\tau}\GL(U_i).\]
Here $F$ is the flavor symmetry.   

At the orthogonal fixed vertex the relevant cocharacter is minuscule; at the symplectic fixed vertex it is quasi-minuscule. The formula above therefore allows the same geometric construction to be carried out in all cases.

\medskip
\noindent\textbf{GKLO homomorphisms and Coulomb branches.}
For the signed-quiver setting, we define explicit generating series $\hat B_i(u)$ and $\hat\Theta_i(u)$ of $q$-difference operators. The shift $\mu$ and the leading coefficients $k_i$ are determined by the quiver and dimension data. Theorem~\ref{thm:gklo} shows that the assignments
\[
B_i(u)\longmapsto\hat B_i(u),\qquad
\Theta_i(u)\longmapsto\hat\Theta_i(u),\qquad
\mathbb K_i\longmapsto k_i
\]
define a homomorphism $\psi$ from the shifted affine $\mathrm{i}$quantum group $\Ui_\mu$, with flavor parameters, to the localized difference algebra. The geometric factorization, Theorem~\ref{thm:Coulomb}, asserts that $\psi$ factors through a homomorphism
\[
\Psi:\Ui_\mu\otimes_{\mathbb C}K^F(\pt)
\longrightarrow
\widetilde{\mathcal A_q}\otimes_{\mathbb C[q^{\pm1}]}\mathbb C(q).
\]
Here $F$ is the flavor group, and $\widetilde{\mathcal A_q}$ denotes the coefficient extension of the Coulomb branch algebra obtained by adjoining the square roots of determinant characters specified in the respective constructions.

The factorization identifies the coefficients of $\hat B_i(u)$ with dressed monopole classes, together with their prescribed character factors and, at the symplectic vertex, terms from the equivariant coefficient ring. The coefficients of $\hat\Theta_i(u)$ also belong to this coefficient ring. The proof thus combines verification of the Drinfeld relations with the minuscule and quasi-minuscule localization calculations. It provides a geometric interpretation of the generators at all three vertex types. In finite ADE type the source is the shifted affine $\mathrm{i}$quantum group of \cite[Definition 2.4]{LWW26}; outside that setting we use the formal extension of its current presentation specified in Section~\ref{sec:iquantum}.

\medskip
\noindent\textbf{Geometric context and related work.}
The constructions above belong to the Coulomb side of the geometric representation theory of quantum algebras. On the Higgs side, the basic mechanism is convolution on correspondences between quiver varieties. Its type $A$ model is the generalized Steinberg variety, whose equivariant $K$-theory realizes affine quantum groups by the work of Ginzburg--Vasserot and Vasserot \cite{GV93,Vas98}. Nakajima's quiver varieties extend this realization to type ADE \cite{Nak01}, while equivariant homology gives the corresponding Yangians \cite{Var00}. For quantum symmetric pairs, generalized Steinberg varieties of type $C$ play an analogous role: their equivariant $K$-theory realizes affine $\mathrm{i}$quantum groups of type AIII \cite{FMX22,SW26,LSX26}, and their equivariant homology realizes twisted Yangians \cite{dong2024flag,SY26}. The present paper studies the Coulomb counterpart of this picture.

The difference-operator description on the Coulomb side has its origin in the representations of Yangians and quantum groups constructed by Gerasimov, Kharchev, Lebedev, and Oblezin \cite{GKLO05,GKLOquant}. Its geometric meaning emerged in the quantization of affine Grassmannian slices \cite{KWWY14}. The identification of quiver Coulomb branches with generalized slices, together with the homomorphism from shifted Yangians to their quantized algebras \cite{BFN19}, placed these operators within the BFN construction; the $K$-theoretic extension involves shifted quantum affine algebras \cite{FT19}. For quivers with involutions, the proposed relation between fixed loci in affine Grassmannian slices and shifted twisted Yangians \cite{LWW,lu2025shifted} suggests a corresponding extension to quantum symmetric pairs; see also \cite{bartlett2025gklo} for type AI. The cohomological realizations for fixed-point-free involutions and orthosymplectic quivers were developed in \cite{SSX26,SSX26b,wang2025quivers}.

GKLO representations for shifted affine $\mathrm{i}$quantum groups have recently been constructed for quasi-split ADE types in \cite{LWW26} and for split simply-laced types in \cite{li2026gklorepresentationsshiftedquantum}. Our formulas are obtained from the monopole geometry of the specified gauge theories. The factorization theorem shows that these formulas define classes in the corresponding Coulomb branch algebras, giving a $K$-theoretic extension of the cohomological constructions in the cases treated here.

The type $\mathrm{AIII}^{(\tau)}_{n}$ specializations are among the cases proposed in \cite[Section~3(ix)]{Nak25}. The signed theories treated here come with an explicit cotangent polarization, so the ordinary BFN construction is available. The orthosymplectic quivers, and the mixed types $D$ and $E$ in \textit{loc.\ cit.}, need not admit such a polarization. A comparable $K$-theoretic realization of non-cotangent matter would require an extension of the geometric framework.

%\noindent\textbf{Organization of the paper.}
The paper is organized as follows. In Section~\ref{sec:Coulomb} we review $K$-theoretic Coulomb branches of cotangent type and prove the quasi-minuscule monopole formula. Section~\ref{sec:quiverinv} introduces signed quivers with isolated fixed vertices, and studies the corresponding quantized K-theoretic Coulomb branches. Section~\ref{sec:gklo} introduces the current presentation of the shifted $\mathrm{i}$quantum groups, constructs their GKLO-type representations and proves the Coulomb branch factorization. Appendices~\ref{sec:appen1}--\ref{sec:appen3} prove Theorem \ref{thm:gklo}.

\subsection*{Acknowledgments}
This work is supported by the National Key R\&D Program of China (No. 2025YFA1017400). YS is supported by the Science and Technology Commission of Shanghai Municipality (No. 22DZ2229014). We used generative AI to polish the article and proofread the text.

\section{Coulomb branches of cotangent type}\label{sec:Coulomb}
In this section, we recall the definition of the Coulomb branch introduced by Braverman, Finkelberg, and Nakajima \cite{NakCou1,BFN18}. We first state the formula for minuscule monopole operators, and then present a formula for quasi-minuscule monopole operators.

\subsection{BFN Coulomb Branch}
Let $G$ be a complex connected reductive group with a Borel subgroup $B$ and a maximal torus $T$. Let $X_*(T)$ be the cocharacter lattice of $T$ with dominant ones denoted by $X_*(T)^+$. Recall the dominance order on $X_*(T)^+$: for any $\lambda,\mu\in X_*(T)^+$, $\lambda\leq \mu$ if and only if $\mu-\lambda$ is a nonnegative integral linear combination of the positive coroots.
Let
$\calO=\mathbb{C}[[z]]$, 
$\calK=\mathbb{C}(\!(z)\!)$. We write $G_\calK:=G(\calK)$, $G_\calO:=G(\calO)$ and $T_\calK:=T(\calK)$, respectively. Any cocharacter $\lambda:\mathbb{G}_m\rightarrow T$ gives a homomorphism $\mathcal{K}^*\rightarrow T_\mathcal{K}$, and we let $z^\lambda$ denote the image of $z\in \mathcal{K}^*$.
Recall the affine Grassmannian is defined to be $\Gr_G=G_\calK/G_\calO$. By the Cartan decomposition,
\[\Gr_G=\bigsqcup_{\lambda\in X_*(T)^+}\Gr_\lambda,\]
where $\Gr_\lambda:=G_\calO z^\lambda G_\calO/G_\calO$. Its closure $\bGr_\lambda=\bigsqcup\limits_{\mu\in X_*(T)^+,\mu\leq \lambda}\Gr_\mu$. Let $P_\lambda$ be the subgroup of $G$ generated by the maximal torus $T$ and the root subgroups $U_\alpha$ satisfying $\langle\lambda,\alpha\rangle\leq 0$. Then $\Gr_\lambda$ is an affine bundle over the partial flag variety $G/P_\lambda$, see \cite{Zhu17}.

Let $\bfN$ be a complex representation of $G$. Let $\bfN_\calK:=\bfN(\calK)$ and $\bfN_\calO:=\bfN(\calO)$, respectively.
We consider
$$\mathcal{T}_{G,\bfN}
=G_\calK\times_{G_\calO} \bfN_\calO
=\{(gG_\calO,x):x\in g\bfN_\calO\}\subset 
\Gr_G\times \bfN_\calK.$$
The BFN space $\calR_{G,\bfN}$ \cite{BFN18} is defined by the pull back square
$$\xymatrix{
\calR_{G,\bfN} \ar@{^{(}->}[r] \ar@{^{(}->}[d] & \Gr_G\times \bfN_\calO \ar@{^{(}->}[d]\\
\mathcal{T}_{G,\bfN} \ar@{^{(}->}[r]&\Gr_G\times \bfN_\calK,}$$
i.e.,
$$\calR_{G,\bfN} = \{(gG_\calO,x): x\in g\bfN_\calO\cap \bfN_\calO\}\subset \Gr_G\times \bfN_\calO.$$

Following \cite[Section 2(i)]{BFN18}, we consider the $\mathbb{C}^\times$ action on $\calR_{G,\bfN}$, which rotates $z\in \calO$ by weight $1$ and scales $\bfN$ by weight $\frac{1}{2}$ simultaneously. 
Let $q^2$ denote the standard representation of $\mathbb{C}^\times$, and denote this $\mathbb{C}^\times$ by $\mathbb{C}^\times_q$. Let $\widetilde{\mathbb{C}}^\times_q$ denote the double cover of $\mathbb{C}^\times_q$. Hence,  
$K_{\widetilde{\mathbb{C}}^\times_q}(\pt)=\mathbb{C}[q^{\pm 1}]$. 
The quantized K-theoretic Coulomb branch algebra is defined to be the complexified equivariant K-homology group
$$\mathcal{A}_q(G,\bfN) = 
K^{G_{\calO}\rtimes\widetilde{\mathbb{C}}^\times_q}(\calR_{G,\bfN}),$$
equipped with the convolution product.

Let $T$ be a maximal torus of $G$, and let $W$ be the Weyl group. Then $\mathcal{A}_q(T,0)$ is generated by the Laurent polynomial ring $\mathbb{C}[T][q^{\pm 1}]$, together with symbols $D_\lambda$ for $\lambda\in X_*(T)$, satisfying the following conditions: 
\[D_\lambda D_\mu=D_{\lambda+\mu}, \quad D_\lambda e^\eta=q^{2\langle\lambda, \eta\rangle}e^\eta D_\lambda,\]
where $\eta\in X^*(T)$. Let $\calD_q(T)$ be the localization of $\mathcal{A}_q(T,0)$ by the multiplicative set generated by $\{1-q^me^\alpha\mid m\in \mathbb{Z},\alpha\in R\}\cup\{1-q^m\mid m\in \mathbb{Z}\setminus\{0\}\}$. Moreover, let
\[\calD_q(T)_{\loc}:=\calD_q(T)\otimes_{\mathbb{C}[q^{\pm 1}]}\mathbb{C}(q).\]

Then we have the following embedding of the Coulomb branch algebra $\mathcal{A}_q(G,\bfN)$ 
\begin{equation}\label{eq:embeddiff}
z^*\circ (\iota_*)^{-1}:\mathcal{A}_q(G,\bfN)\hookrightarrow \calD_q(T). 
\end{equation}

\subsection{Minuscule monopole operators}

Recall there is a natural projection 
$$\pi:\calR_{G,\bfN}\longrightarrow \Gr_G.$$
For any $\lambda\in X_*(T)^+$, let us denote 
$$\calR_{\lambda}=\pi^{-1}(\Gr_\lambda),\qquad \text{ and }\qquad \calR_{\leq \lambda} = \pi^{-1}(\bGr_\lambda).$$
Note that we can identify
$K^{G_\calO}(\calR_\lambda) \simeq K^{G_\calO}(\Gr_\lambda) \simeq \mathbb{C}[T]^{W_\lambda}$, where $W_\lambda$ is the stabilizer group of $\lambda$ in the Weyl group $W$. 

For a finite dimensional $T\times\widetilde{\mathbb{C}}^\times_q$-vector space $V$, we let 
\[\Eu(V):=\sum_i(-1)^i\wedge^iV^*=\prod_{w}(1-e^{-w}),\] where the product is over all the $T\times\widetilde{\mathbb{C}}^\times_q$-weights (with multiplicities) of $V$.

Now assume that $\lambda\in X_*(T)^+$ is minuscule, i.e.,
$\langle \lambda,\alpha\rangle\in \{0,1\}$
for all positive roots $\alpha$. Then the minuscule monopole operator is given by the following formula:
\begin{equation}\label{eq:minu-mono-op}
z^*\circ (\iota_*)^{-1}(f[\calO_{\mathcal{R}_\lambda}])=\sum_{w\in W^\lambda}
w\left(f\cdot \frac{\Eu\left(
z^\lambda\bfN_\calO/(z^\lambda\bfN_\calO\cap\bfN_\calO)\right)}{
\Eu(T_{\lambda}\Gr_\lambda)
}\right)
D_{w\lambda}\in \calD_q(T).
\end{equation}
Here $f\in \mathbb{C}[T]^{W_\lambda}\simeq K^{G_\calO}(\calR_\lambda)$, $W^\lambda$ is the set of minimal length representatives of cosets in $W/W_\lambda$, and $T_{\lambda}\Gr_\lambda
$ is the tangent space of $\Gr_\lambda$ at the torus fixed point $z^\lambda G_\calO/G_\calO$. Note that we have a bijection $W^\lambda\cong W\cdot \lambda$ by $w\mapsto w\lambda$. 
For any $\mu\in W\cdot\lambda$, 
\begin{equation}\label{equ:wts}
    T_\mu\Gr_\lambda = \bigoplus_{\alpha\in R, \langle\mu,\alpha\rangle\geq 1}\bigoplus_{n=0}^{\langle\mu,\alpha\rangle - 1}\mathfrak{g}_\alpha z^n,
\end{equation}
where $R$ is the set of roots of $G$.

If $N$ is one dimensional with $T$-weight $\mu$, then 
$$z^\lambda N_\calO=N\otimes z^m\calO,$$
where $m=\langle \lambda, \mu\rangle$. Thus,
$$z^{\lambda}N_\calO/(z^{\lambda}N_\calO\cap N_\calO)
=\begin{cases}
0, & \textit{ if }m\geq 0,\\
N\otimes(\mathbb{C}z^{-1}\oplus \cdots \oplus \mathbb{C}z^{m}), & \textit{ if } m< 0,
\end{cases}$$
and
\begin{equation}\label{equ:EuN}
    \Eu(z^{\lambda}N_\calO/(z^{\lambda}N_\calO\cap N_\calO))
=\begin{cases}
1 , & \textit{ if } m\geq 0,\\
(1-qe^{-\mu})(1-q^3e^{-\mu})\cdots (1-q^{-2m-1}e^{-\mu})

, & \textit{ if } m< 0.
\end{cases}
\end{equation}
Recall that $N$ is assumed to have $\widetilde{\mathbb{C}}^\times_q$-weight $q$.

\subsection{Quasi-minuscule monopole operators}
In this section, we compute the monopole operators associated to a quasi-minuscule cocharacter. The cohomology case is studied in \cite{SSX26b}.

Recall that a cocharacter $\lambda$ is called quasi-minuscule if $\lambda=\theta^\vee$ with $\theta$ being a highest root of $G$ (when the root system is irreducible, $\theta$ is unique). For any positive root $\alpha\in R^+\setminus\{\theta\}$, $\langle\lambda, \alpha\rangle\in \{0,1\}$, and $0$ is the only element in $X_*(T)^+$ smaller than $\lambda$. Hence, 
\[\bGr_\lambda=\Gr_\lambda\sqcup \{e\},\]
where $e\in \Gr_G$ is the identity point. Moreover, $\Gr_\lambda$ is a line bundle over $G/P_\lambda$. By a computation of the tangent weights, 
\[\Gr_\lambda\simeq G\times_{P_\lambda}\mathbb{C}_\theta\]
as $G$-equivariant line bundles over $G/P_\lambda$. Here $\mathbb{C}_\theta$ is the one-dimensional representation of $P_\lambda$. We let $\mathcal{L}$ denote the line bundle on the right hand side. The explicit isomorphism $\mathcal{L}\simeq \Gr_\lambda$ is given by 
\[[g,t]\mapsto gx_\theta(tz)z^\lambda G_\calO/G_\calO,\]
where $x_\theta(t):=\exp(tE_\theta)$ for $E_\theta$ the root vector in $\mathfrak{g}_\theta$.

There is a resolution of $\overline{\Gr_\lambda}$ given by 
\[p:\mathbb{P}(\calO\oplus\mathcal{L})\rightarrow \overline{\Gr_\lambda},\]
where the fibers at infinity all map to the identity point $e\in \overline{\Gr_\lambda}$, see \cite{NP01}. 

Note that when restricting to $G\times_{P_\lambda} \mathbb{C}_\theta$, the fiberwise quotient 
$p^*(\mathcal{T}_{G,\bfN})/p^*(\mathcal{R}_{G,\bfN})$ is a vector bundle of rank 
$d:=\dim(z^{\lambda}\bfN_\calO/(z^{\lambda}\bfN_\calO\cap \bfN_\calO))$; 
    at the infinity points, the fiber of $p^*(\mathcal{R}_{G,\bfN})$ equals the fiber of $p^*(\mathcal{T}_{G,\bfN})$.

\begin{lemma}\cite[Lemma 2.2]{SSX26b}
There exists a unique $G\times \widetilde{\mathbb{C}}^\times_q$-equivariant subbundle $\tilde{\mathcal{R}}_{G,\bfN,\leq \lambda}$ of $p^*(\mathcal{T}_{G,\bfN})$ 
such that 
\begin{itemize}
    \item $\tilde{\mathcal{R}}_{G,\bfN,\leq \lambda}$ coincides with $p^*(\mathcal{R}_{G,\bfN})$ over $p^{-1}(\Gr_\lambda)=G\times_{P_\lambda}\mathbb{C}_\theta$; 
    \item the fiberwise quotient $p^*(\mathcal{T}_{G,\bfN})/\tilde{\mathcal{R}}_{G,\bfN,\leq \lambda}$
    is a vector bundle of rank $d$. 
\end{itemize}
\end{lemma}
The map $p$ also induces a proper map $\tilde{p}:\tilde{\calR}_{G,\bfN,\leq \lambda}\rightarrow \mathcal{T}_{G,\bfN}$, with image lying in $\mathcal{R}_{G,\bfN,\leq \lambda}:=\pi^{-1}(\bGr_\lambda)\subset \calR_{G,\bfN}$. 
We have the following diagram
$$\xymatrix@!C=4pc{
p^*(\mathcal{T}_{G,\bfN})\ar[rr]\ar[dd]
&& \mathcal{T}_{G,\bfN,\leq \lambda}\ar[dd]|{\hole}\ar[rr]
&& \mathcal{T}_{G,\bfN}\ar[dd]|{\hole}\\
& \tilde{\mathcal{R}}_{G,\bfN,\leq \lambda}\ar[ul]\ar[dl]
\ar[rr]^{\tilde{p}\qquad\quad } && \mathcal{R}_{G,\bfN,\leq \lambda}\ar[ul]\ar[dl]\ar[rr]&&
\mathcal{R}_{G,\bfN}.\ar[ul]\ar[dl]^\pi\\
\mathbb{P}(\mathcal{O}\oplus \mathcal{L})
\ar[rr]^-{p} &&\overline{\Gr_\lambda}\ar[rr]&&\Gr_G}$$

For any $f\in K^G(\Gr_\lambda)\simeq K^G(G/P_\lambda)\simeq \mathbb{C}[T]^{W_\lambda}$, we can pull it back to $\tilde{\calR}_{G,\bfN,\leq \lambda}$ via the projection map $\tilde{\calR}_{G,\bfN,\leq \lambda}\rightarrow\mathbb{P}(\calO\oplus\mathcal{L})\rightarrow G/P_\lambda$. Therefore, we can consider the class
\[\tilde{p}_*(f[\calO_{\tilde{\calR}_{G,\bfN,\leq \lambda}}])\in \mathcal{A}_q(G,\bfN), \]
and we have the following formula.

\begin{theorem}\label{thm:quasi}
    With the above notations,
    \begin{align*}
    &z^*\circ (\iota_*)^{-1}\bigg(\tilde{p}_*(f[\calO_{\tilde{\calR}_{G,\bfN,\leq \lambda}}])\bigg)\\
    =&\sum_{w\in W^\lambda}w\left(f\cdot \frac{\Eu\left(z^\lambda\bfN_\calO/(z^\lambda\bfN_\calO\cap\bfN_\calO)\right)}{\Eu(T_{\lambda}\Gr_\lambda)}\right)(D_{w\lambda}-1)+\cdots\in \calD_q(T),
\end{align*}
where $\cdots$ denotes some polynomial in $K^{G\times\widetilde{\mathbb{C}}^\times_q}(\pt)$. In particular, 
\[\sum_{w\in W^\lambda}w\left(f\cdot \frac{\Eu\left(z^\lambda\bfN_\calO/(z^\lambda\bfN_\calO\cap\bfN_\calO)\right)}{\Eu(T_{\lambda}\Gr_\lambda)}\right)(D_{w\lambda}-1)\in \mathcal{A}_q(G,\bfN).\]
\end{theorem}
\begin{proof}
    The proof is similar to the proof of \cite[Theorem 2.4]{SSX26b}. The torus fixed points in $\tilde{\calR}_{G,\bfN,\leq \lambda}$ are $\{(w\cdot P_\lambda, 0), (w\cdot P_\lambda, \infty)\mid w\in W^\lambda\}$. Since everything is $G$-equivariant, we only need to compute the localization contribution when $w\in W^\lambda$ is the identity element. By definition, 
    \[\begin{aligned}
    \Eu(T_{(1\cdot P_\lambda, 0)}\mathbb{P}(\calO\oplus\mathcal{L}))
    & =(1-q^{-2}e^{-\theta})\cdot \Eu(T_{1\cdot P_\lambda}G/P_\lambda)=\Eu(T_\lambda\Gr_\lambda),\\
    \Eu(T_{(1\cdot P_\lambda, \infty)}\mathbb{P}(\calO\oplus\mathcal{L}))
    & =(1-q^2e^\theta)\cdot \Eu(T_{1\cdot P_\lambda}G/P_\lambda).
    \end{aligned}\]
    
    Hence, by localization we get
    \begin{align*}
        &z^*\circ (\iota_*)^{-1}\big(\tilde{p}_*(f[\calO_{\tilde{\calR}_{G,\bfN,\leq \lambda}}])\big)\\
        =&\sum_{w\in W^\lambda}w\bigg(f\frac{\Eu\big((p^*(\mathcal{T}_{G,\bfN})/\tilde{\calR})_{(1\cdot P_\lambda, 0)}\big)}{\Eu(T_{(1\cdot P_\lambda, 0)}\mathbb{P}(\calO\oplus\mathcal{L}))}\bigg)D_{w\lambda}+\sum_{w\in W^\lambda}w\bigg(f\frac{\Eu\big((p^*(\mathcal{T}_{G,\bfN})/\tilde{\calR})_{(1\cdot P_\lambda, \infty)}\big)}{\Eu(T_{(1\cdot P_\lambda,\infty)}\mathbb{P}(\calO\oplus\mathcal{L}))}\bigg)\\
        =&\sum_{w\in W^\lambda}w\bigg(f\frac{\Eu\big(z^\lambda\bfN_\calO/(z^\lambda\bfN_\calO\cap\bfN_\calO)\big)}{\Eu(T_\lambda\Gr_\lambda)}\bigg)(D_{w\lambda}-1)\\
        &+\sum_{w\in W^\lambda}w\bigg(f\frac{\Eu\big(z^\lambda\bfN_\calO/(z^\lambda\bfN_\calO\cap\bfN_\calO)\big)}{(1-q^{-2}e^{-\theta})\Eu(T_{1\cdot P_\lambda}G/P_\lambda)}+f\frac{\Eu\big((p^*(\mathcal{T}_{G,\bfN})/\tilde{\calR})_{(1\cdot P_\lambda, \infty)}\big)}{(1-q^{2}e^{\theta})\Eu(T_{1\cdot P_\lambda}G/P_\lambda)}\bigg).
    \end{align*}
Suppose the $T\times \widetilde{\mathbb{C}}^\times_q$-weights in $z^\lambda\bfN_\calO/(z^\lambda\bfN_\calO\cap\bfN_\calO)$ are $\{e^{\gamma_1},\cdots, e^{\gamma_d}\}$. 
The fiber of $\mathbb{P}(\mathcal{O}\oplus \mathcal{L})$ over $1\cdot P_\lambda$ is isomorphic to $\mathbb{P}^1$. Restricting the vector bundle $p^*(\mathcal{T}_{G,\bfN})/\tilde{\mathcal{R}}$ to this fiber and applying \cite[Lemma 2.5]{SSX26b}, we obtain that
the $T\times \widetilde{\mathbb{C}}^\times_q$-weights in $\big(p^*(\mathcal{T}_{G,\bfN})/\tilde{\calR}\big)_{(1\cdot P_\lambda, \infty)}$ can be ordered as $e^{\gamma_1'},\ldots,e^{\gamma_d'}$ with $e^{\gamma_i'}=e^{\gamma_i}e^{n_i\theta}q^{2n_i}$ for some $n_i\in \mathbb{Z}$. Therefore, 
\[F:=f\frac{\Eu\big(z^\lambda\bfN_\calO/(z^\lambda\bfN_\calO\cap\bfN_\calO)\big)}{1-q^{-2}e^{-\theta}}+f\frac{\Eu\big((p^*(\mathcal{T}_{G,\bfN})/\tilde{\calR})_{(1\cdot P_\lambda, \infty)}\big)}{1-q^2e^\theta}\in K^T(\pt)[q^{\pm 1}].\]
Moreover, $F$ is $W_\lambda$ invariant, since $f$ and $e^\theta$ are invariant, as are the two Euler classes
\[
 \Eu\big(z^\lambda\bfN_\calO/(z^\lambda\bfN_\calO\cap\bfN_\calO)\big)
 \quad\text{and}\quad
 \Eu\big((p^*(\mathcal{T}_{G,\bfN})/\tilde{\calR})_{(1\cdot P_\lambda,\infty)}\big).
\]
Hence, by the localization theorem on $G/P_\lambda$, we get that
\begin{align*}
    &\sum_{w\in W^\lambda}w\bigg(f\frac{\Eu\big(z^\lambda\bfN_\calO/(z^\lambda\bfN_\calO\cap\bfN_\calO)\big)}{(1-q^{-2}e^{-\theta})\Eu(T_{1\cdot P_\lambda}G/P_\lambda)}+f\frac{\Eu\big((p^*(\mathcal{T}_{G,\bfN})/\tilde{\calR})_{(1\cdot P_\lambda, \infty)}\big)}{(1-q^{2}e^{\theta})\Eu(T_{1\cdot P_\lambda}G/P_\lambda)}\bigg)\\
    =&\sum_{w\in W^\lambda}w\bigg(\frac{F}{\Eu(T_{1\cdot P_\lambda}G/P_\lambda)}\bigg)\in K^{G\times\widetilde{\mathbb{C}}^\times_q}(\pt).
\end{align*}
This finishes the proof of the theorem.
\end{proof}

\subsection{Flavor symmetry}
Let $F$ be a reductive group, and we assume further that the $G$-representation $\bfN$ can be extended to a $G\times F$-representation. 
Then we can extend the above definition by 
$$\mathcal{A}_q(G\times F,\bfN) := K^{(G_{\calO}\times F_{\calO})\rtimes\widetilde{\mathbb{C}}^\times_q}(\calR_{G,\bfN}).$$
The embedding \eqref{eq:embeddiff} becomes
$$\mathcal{A}_q(G\times F,\bfN)\hookrightarrow K^F(\pt)\otimes \calD_q(T).$$
The results in this section still hold.

\section{Quantized Coulomb branch for signed quivers with involutions}\label{sec:quiverinv}
In this section, we introduce signed quivers with involutions, and study the corresponding quantized Coulomb branches. 

\subsection{Quivers, signs, and the cotangent polarization}\label{sec:quiver}
Let $Q=(Q_0,Q_1,s,t)$ be a finite simply-laced quiver without edge loops. We assume that there is an involution $\tau$ on $Q$ satisfying
\begin{equation*}
 s(\tau h)=\tau(t(h)),\qquad t(\tau h)=\tau(s(h)).
\end{equation*}
Let $Q_0^\tau$ denote the set of fixed vertices.
\begin{lemma}
    The fixed vertices are pairwise nonadjacent.
\end{lemma}
\begin{proof}
    Suppose $i,j\in Q_0^\tau$, and $h\in Q_1$ is an edge going from $i$ to $j$. Then by the above condition, $\tau h$ would be an edge going from $j$ to $i$. This contradicts the condition that $Q$ is simply-laced. 
\end{proof}
Choose a
decomposition
\[
 Q_0=Q_0^+\sqcup Q_0^\tau\sqcup\tau(Q_0^+),
 \qquad Q_0^\tau=Q_O\sqcup Q_S.
\] 
Choose a subset $Q_1^+\subset Q_1$ of representatives of the two-element arrow orbits.
Write $Q_1^\tau$ for the fixed arrows. Since the underlying graph is
simply-laced, an arrow is fixed precisely when its endpoints are
interchanged by $\tau$.

For $i\in Q_0^+$ choose nonzero $V_i$ and put $V_{\tau i}=V_i^*$. At $i\in Q_O$
choose an orthogonal space $V_i$ of dimension at least three, and at
$i\in Q_S$ choose a symplectic space of positive even dimension. Put
$v_i=\dim V_i=v_{\tau i}$ and
\begin{equation}\label{eq:signedgauge}
 G=\prod_{i\in Q_0^+}\GL(V_i)\times
   \prod_{i\in Q_O}\SO(V_i)\times
   \prod_{i\in Q_S}\Sp(V_i).
\end{equation}
For each vertex $i$, we choose a framing space $W_i$. If $i$ is fixed, we require $W_i$ to carry a bilinear form of the type opposite to that on $V_i$ and to admit a polarization $W_i=U_i\oplus U_i^*$. In particular, $w_i=\dim W_i$ is even at every fixed vertex.
Set
\[
 m_i:=\begin{cases}w_i,&i\notin Q_0^\tau,\\
 \dim U_i=w_i/2,&i\in Q_0^\tau,\end{cases}
 \qquad
 F=\prod_{i\notin Q_0^\tau}\GL(W_i)
   \times\prod_{i\in Q_0^\tau}\GL(U_i).
\]
The chosen half of the matter representation is
\begin{equation}\label{eq:signedmatter}
\begin{split}
 \bfN={}&\bigoplus_{h\in Q_1^+}\Hom(V_{s(h)},V_{t(h)})
 \oplus\bigoplus_{h\in Q_1^\tau}\Lambda^2 V_{t(h)}\\
 &\oplus\bigoplus_{i\notin Q_0^\tau}\Hom(V_i,W_i)
 \oplus\bigoplus_{i\in Q_0^\tau}U_i\otimes V_i.
\end{split}
\end{equation}
The symplectic matter space is then $\mathbf M=\bfN\oplus\bfN^*$. At a fixed vertex this polarization restricts the
framing symmetry to $\GL(U_i)$. In particular, the even-dimensional
orthogonal framing at a symplectic gauge vertex makes this a
cotangent-type theory.

Here is a signed fixed-locus description of \eqref{eq:signedmatter}.
Choose $\varepsilon_i=\varepsilon_{\tau i}$, equal to $1$ on nonfixed
pairs and orthogonal vertices and to $-1$ on symplectic vertices.
For a vector space $U$ and $\varepsilon\in\{1,-1\}$, the hyperbolic pairing of sign $\varepsilon$ on $U\oplus U^*$ is
\[
 H_\varepsilon\bigl((u,\phi),(v,\psi)\bigr)
 =\psi(u)+\varepsilon\phi(v).
\]
It is nondegenerate and satisfies
$H_\varepsilon(y,x)=\varepsilon H_\varepsilon(x,y)$, so it is
symmetric for $\varepsilon=1$ and alternating for $\varepsilon=-1$.
Both $U$ and $U^*$ are isotropic, and their mutual pairing is the
evaluation pairing. Equip $V_i\oplus V_{\tau i}=V_i\oplus V_i^*$
with $H_1$ for $i\in Q_0^+$. Write $B_i$ for the pairing between
$V_i$ and $V_{\tau i}$, including the chosen form when $i=\tau i$.
The adjoint $f^\dagger:V_{\tau j}\to V_{\tau i}$ of
$f:V_i\to V_j$ is characterized by
\[
 B_j(fx,y)=B_i(x,f^\dagger y),\qquad
 (f^\dagger)^\dagger=\varepsilon_i\varepsilon_j f.
\]
Choose arrow signs $\{\eta_h\in \{\pm 1\}\mid h\in Q_1\}$ satisfying
\begin{equation}\label{eq:arrowsigns}
 \eta_h\eta_{\tau h}=\varepsilon_{s(h)}\varepsilon_{t(h)},
 \qquad \eta_h=-1\quad(h\in Q_1^\tau).
\end{equation}
Then $(\sigma f)_{\tau h}=\eta_h f_h^\dagger$ is an involution on
$E_V=\bigoplus_h\Hom(V_{s(h)},V_{t(h)})$. Let us consider the fixed locus $E_V^\sigma$. If an arrow $h$ is not fixed by $\tau$, the $ \sigma$-fixed locus in the two summands corresponding to $h$ and $\tau h$ are identified with the $h$-summand. Therefore,
projection to $Q_1^+$ identifies each two-element fixed summand with
the corresponding Hom space in \eqref{eq:signedmatter}. A fixed arrow has condition
$f_h=-f_h^\dagger$, giving $\Lambda^2 V_{t(h)}$.

For the fixed framing summand use
$\Hom(V_i,U_i)\oplus\Hom(U_i^*,V_i)$, with the hyperbolic
pairing of sign $-\varepsilon_i$ on the framing vector space $W_i=U_i\oplus U_i^*$. The involution
$(a,b)\mapsto(b^\dagger,-a^\dagger)$ has fixed locus
$\{(a,-a^\dagger)\}\cong U_i\otimes V_i$.
For a nonfixed vertex $i$, introduce two new framing vertices
$p_i,p_i^\vee$ carrying $W_i,W_i^*$, respectively, with the pairing
$H_1$. Extend $\tau$ by $\widetilde\tau(p_i)=p_i^\vee$ and
$\widetilde\tau(p_i^\vee)=p_i$. Adjoin the two arrows
\[
 i\longrightarrow p_i,\qquad p_i^\vee\longrightarrow\tau i,
\]
which form a single orbit under $\widetilde\tau$. Their representation
space and its signed involution are
\[
 \Hom(V_i,W_i)\oplus\Hom(W_i^*,V_{\tau i}),
 \qquad \sigma(a,b)=(-b^\dagger,-a^\dagger).
\]
Since adjoint twice is the identity here, this is an involution with
fixed locus
\[
 \{(a,-a^\dagger):a\in\Hom(V_i,W_i)\}
 \cong\Hom(V_i,W_i).
\]
The flavor factor $\GL(W_i)$ acts on $W_i^*$ contragrediently, making
this identification equivariant.

We perform this construction independently for every
$i\notin Q_0^\tau$. In particular, $p_{\tau i}$ is a new
vertex distinct from $p_i^\vee$: its space is the independently
chosen $W_{\tau i}$, not $W_i^*$. Thus a nonfixed gauge orbit gives
two independent framing-arrow orbits, whose fixed loci are
$\Hom(V_i,W_i)$ and $\Hom(V_{\tau i},W_{\tau i})$. This explains
why the third summand of \eqref{eq:signedmatter} runs over all
nonfixed vertices, rather than over $Q_0^+$ alone. It is the fixed
locus of the augmented framing-arrow representation under
the signed involution $\sigma$ induced by $\widetilde\tau$; we are
not taking the fixed part of the original sum
$\bigoplus_i\Hom(V_i,W_i)$, on which no such involution has been
specified.

The gauge group is 
\begin{equation*}
 G_V=\prod_{i\in Q_0\setminus Q_0^\tau}\GL(V_i)\times
   \prod_{i\in Q_O}\mathrm{O}(V_i)\times
   \prod_{i\in Q_S}\Sp(V_i).
\end{equation*}
The group $G$ in \eqref{eq:signedgauge} can be identified with the identity component of the fixed gauge group 
\[G_V^\sigma:=\{(g_i)\in G_V\mid g_{\tau i}^{-1} = g_i^\dagger, \forall i\in Q_0\}.\]

\subsection{Monopole operators}

In this section, we study the K-theoretic quantized Coulomb branch
\[
 \mathcal A_q(G\times F,\bfN)=K^{(G_\calO\times F_\calO)\rtimes
 \widetilde{\mathbb C}_q^\times}(\calR_{G,\bfN})
\]
associated to the following $(G,\bfN,F)$ from Section \ref{sec:quiver}:
\begin{equation*}
 G=\prod_{i\in Q_0^+}\GL(V_i)\times
   \prod_{i\in Q_O}\SO(V_i)\times
   \prod_{i\in Q_S}\Sp(V_i),
\end{equation*}
\begin{equation*}
 \bfN=\bigoplus_{h\in Q_1^+}\Hom(V_{s(h)},V_{t(h)})
 \oplus\bigoplus_{h\in Q_1^\tau}\Lambda^2 V_{t(h)}\oplus\bigoplus_{i\notin Q_0^\tau}\Hom(V_i,W_i)
 \oplus\bigoplus_{i\in Q_0^\tau}U_i\otimes V_i,
\end{equation*}
\[F=\prod_{i\notin Q_0^\tau}\GL(W_i)
   \times\prod_{i\in Q_0^\tau}\GL(U_i).\]
We give some explicit formulae for the monopole operators.

For the group $\GL(V_i), \SO(V_i)$ or $\Sp(V_i)$, we use  $T_{V_i}$ to denote its diagonal maximal torus, and let 
\[
T=\prod_{i\in Q_0^+\sqcup Q_0^\tau}T_{V_i}
\]
be the maximal torus in $G$. We use $x_{i,j}$ (resp.$\epsilon_{i,j}$) to denote the $j$-th coordinate character (resp. cocharacter) for $T_{V_i}$, and let $D_{i,j}\in \mathcal{A}_q(T,0)$ denote the difference operator corresponding to the cocharacter $\epsilon_{i,j}$.

For any $i\in Q_0^+$, let
\[
 x_{\tau i,r}=x_{i,r}^{-1},\qquad
 D_{\tau i,r}=D_{i,r}^{-1}.
\]
For a fixed vertex $i\in Q_0^\tau$, put $r'=v_i+1-r$ and choose coordinates so that
\begin{equation*}
 x_{i,r'}=x_{i,r}^{-1},\qquad D_{i,r'}=D_{i,r}^{-1}.
\end{equation*}
Only $1\leq r\leq\lfloor v_i/2\rfloor$ are independent there.
If $v_i$ is odd, the middle weight is $1$ and has no associated
nonzero coordinate cocharacter. Define the index set
\[
 I_i=\begin{cases}
 \{1,\ldots,v_i\},&i\notin Q_0^\tau\text{ or }i\in Q_S,\\
 \{r:1\leq r\leq v_i,\ r\ne r'\},&i\in Q_O.
 \end{cases}
\]
For $i,j\in Q_0$ and $r\in I_i, s\in I_j$, we have the following commutation rules:
\[D_{i,r}x_{j,s}=\begin{cases}q^{2(\delta_{ij}-\delta_{i,\tau j})\delta_{rs}}x_{j,s}D_{i,r},& \textit{ if both } i,j\notin Q_0^\tau,\\
q^{2(\delta_{rs}-\delta_{rs'})\delta_{ij}}x_{j,s}D_{i,r},& \textit{ if } i,j\in Q_0^\tau,\\
x_{j,s}D_{i,r},& \textit{ others}.
\end{cases} \]
Write $w_{i,b}$ for the torus characters of $W_i$ or $U_i$, according
to the vertex type; there are $m_i$ such characters.

Recall that localization gives an embedding:
\[z^*\circ (\iota_*)^{-1}:\mathcal A_q(G\times F, \bfN)\hookrightarrow K^F(\pt)\otimes_{\mathbb{C}}\calD_q(T).\]

\subsubsection{General linear case}
In this subsection, we take $i\in Q_0^+$. The dominant coweight $\epsilon_{i,1}\in X_*(T)$ is a minuscule coweight for $G$. Hence, the spherical Schubert cell $\Gr_{\epsilon_{i,1}}$ is closed and isomorphic to the projective space $\mathbb{P}^{v_i-1}$. More precisely, it is identified with the moduli space of $\calO$-modules $L$ such that 
\[z\calO\otimes V_i\subset L\subset \calO\otimes V_i,\quad \dim_\mathbb{C}\calO\otimes V_i/L =1.\]
There is a tautological line bundle on $\Gr_{\epsilon_{i,1}}$ whose fiber at $L$ is $\calO\otimes V_i/L$. Thus, the torus weight of this line bundle at the fixed point $z^{\epsilon_{i,r}}$ is $x_{i,r}$. Let $\mathcal{Q}_i$ denote the pullback of this line bundle to $\calR_{\epsilon_{i,1}}$.  
\begin{prop}\label{prop:fmono}
Let $f\in \mathbb{C}[x^{\pm 1}]$ be a Laurent polynomial in one variable. We have 
\begin{align*}
    z^*&\circ(\iota_*)^{-1}\big(f(\mathcal{Q}_i) \cdot [\calO_{\mathcal{R}_{\epsilon_{i,1}}}]\big)
    =\sum_{r=1}^{v_i}f(x_{i,r})
    \prod_{h\in Q_1\setminus Q_1^\tau\atop s(h)=i}\prod_{a=1}^{v_{t(h)}}(1-qx_{i,r}/x_{t(h),a})\\
    &\cdot \prod_{h\in Q_1^\tau\atop s(h)=i}\prod_{1\leq a\leq v_{t(h)}\atop a\neq r}(1-qx_{i,r}/x_{\tau i,a})\cdot \frac{\prod_{1\leq a\leq w_{i}}(1-qx_{i,r}/w_{i,a})}{\prod_{1\leq s\leq v_i, s \neq r}(1-x_{i,s}/x_{i,r})}D_{i,r}
    \in K^F(\pt)\otimes_{\mathbb{C}}\calD_q(T).
\end{align*}
\end{prop}
\begin{proof}
For $1\leq r\leq v_i$, \eqref{equ:wts} gives
$$\Eu(T_{\epsilon_{i,r}} \Gr_{\epsilon_{i,1}})
=\prod_{1\leq s\leq v_i, s \neq r}(1-x_{i,s}/x_{i,r}).$$
Let us compute 
$\Eu\big(z^{\epsilon_{i,r}}\bfN_{\calO}/(z^{\epsilon_{i,r}}\bfN_{\calO}\cap \bfN_{\calO})\big)$ by \eqref{equ:EuN}. Recall $\bfN$ from \eqref{eq:signedmatter}
and $\tilde{\mathbb{C}}^*_q$ acts on $\bfN$ by weight $q$. Since $i\in Q_0^+$, only the first three factors will have contributions, and they are given by the following formulae respectively:
\begin{align*}
& \quad \prod_{h\in Q_1^+\atop s(h)=i}
\prod_{b=1}^{v_{t(h)}}
(1-qx_{i,r}/x_{t(h),b})
\prod_{h\in Q_1^+\atop t(h)=\tau i}
\prod_{a=1}^{v_{s(h)}}
(1-qx_{s(h),a}/x_{\tau i,r})=\prod_{h\in Q_1\setminus Q_1^\tau\atop s(h)=i}\prod_{a=1}^{v_{t(h)}}
(1-qx_{i,r}/x_{t(h),a}),
\end{align*}
$$\prod_{h\in Q_1^\tau\atop s(h)=i}\prod_{1\leq a\leq v_{t(h)}\atop a\neq r}(1-qx_{i,r}/x_{\tau i,a}),\quad\textit{ and } \prod_{1\leq a\leq w_{i}}(1-qx_{i,r}/w_{i,a}).$$
Substituting all of these into \eqref{eq:minu-mono-op}, we get the desired formula. 
\end{proof}

Now let us consider the case of $\lambda=-\epsilon_{i,v_i}\in X_*(T)$. This is also a minuscule coweight, and the spherical Schubert cell $\Gr_{-\epsilon_{i,v_i}}$ is also isomorphic to the projective space $\mathbb{P}^{v_i-1}$. More precisely, it is identified with the moduli space of $\calO$-modules $L$ such that 
\[\calO\otimes V_i\subset L\subset z^{-1}\calO\otimes V_i,\quad \dim_\mathbb{C}L/(\calO\otimes V_i) =1.\]
There is a tautological line bundle on $\Gr_{-\epsilon_{i,v_i}}$ whose fiber at $L$ is $L/(\calO\otimes V_i)$. Therefore, the torus weight of the fiber at the torus fixed point $z^{-\epsilon_{i,r}}$ is $q^{-2}x_{i,r}$. Let $\mathcal{S}_i$ denote the pullback of this line bundle to $\calR_{-\epsilon_{i,v_i}}$. Similar to the previous proposition, we have the following formula, whose proof is almost the same. 
\begin{prop}\label{prop:fmono2}
Let $f\in \mathbb{C}[x^{\pm 1}]$ be a Laurent polynomial in one variable. We have 
\begin{align*}
    z^*&\circ(\iota_*)^{-1}\big(f(\mathcal{S}_i)\cdot [\mathcal{O}_{\mathcal{R}_{-\epsilon_{i,v_i}}}]\big)
    =\sum_{r=1}^{v_i}f(q^{-2}x_{i,r})
    \prod_{h\in Q_1\setminus Q_1^\tau \atop s(h)=\tau i}\prod_{b=1}^{v_{t(h)}}(1-qx_{\tau i,r}/x_{t(h),b})\\
    &\cdot\prod_{h\in Q_1^\tau\atop t(h)=i}\prod_{1\leq a\leq v_{t(h)}\atop a\neq r}(1-qx_{i,a}^{-1}x_{i,r}^{-1})\cdot \frac{\prod_{1\leq a\leq w_{ \tau i}}(1-qx_{\tau i,r}/w_{\tau i,a})}{\prod_{1\leq s\leq v_i, s \neq r}(1-x_{i,r}/x_{i,s})}D_{i,r}^{-1}\in K^F(\pt)\otimes_{\mathbb{C}}\calD_q(T).
\end{align*}
\end{prop}

\subsubsection{Affine Grassmannian for orthogonal and symplectic groups.}\label{sec:orthsympAff}
In this section, we recall some basic facts about the affine Grassmannian for the orthogonal and symplectic groups.
Let $V$ be an orthogonal or symplectic vector space, and let $G:=\SO(V)$ if $V$ is orthogonal or $G:=\Sp(V)$ if $V$ is symplectic. Throughout this paper, we assume $\dim V\geq 3$ if $V$ is an orthogonal vector space. Pick a maximal torus $T$ in $G$, and let $\epsilon_r$ be the $r$-th coordinate cocharacter for $T$.

For a lattice $L$ in $V(\calK)$, we define its dual $$L^\vee:=\{x\in V(\calK)\mid (x,L)\subseteq \calO\},$$ where the pairing $(-,-)$ is induced from the one on $V$. Then the affine Grassmannian $\Gr_{G}$ is the set of self-dual lattices in $V(\calK)$, and the Schubert cell
\begin{align*}
    \Gr_{G,\epsilon_{1}}=&\{\textit{self-dual lattices }  L\subset V(\calK)\mid\\
    &zV(\calO)\subset L\subset z^{-1}V(\calO), \dim_{\mathbb{C}} \frac{V(\calO)}{V(\calO)\cap L}=1\}.
\end{align*}
There is a tautological line bundle $\mathcal{L}$ on $\Gr_{G,\epsilon_{1}}$, whose fiber at a lattice $L$ is $\frac{V(\calO)}{V(\calO)\cap L}$. Under the isomorphism $$K^{G}(\Gr_{G,\epsilon_{1}})\simeq K^{G}(G/P_{\epsilon_{1}})\simeq \mathbb{C}[T]^{W_{\epsilon_{1}}},$$ 
the line bundle $\mathcal{L}$ corresponds to $x_{1}$. The torus fixed points in $\Gr_{G,\epsilon_{1}}$ are $$\big\{z^{\pm\epsilon_{r}}G(\calO)/G(\calO)\mid 1\leq r\leq \lfloor\frac{\dim V}{2}\rfloor\big\}.$$ To simplify notation, we will just use $z^{\pm\epsilon_{r}}$ to denote the corresponding fixed points.

\subsubsection{Orthogonal case}
In this section, $i\in Q_O$.
The cocharacter $\epsilon_{i,1}$ is a minuscule coweight for $\SO(V_i)$. We recall some notations from Section \ref{sec:orthsympAff}.
Let $\mathcal{L}_i$ denote
the tautological line bundle on $\Gr_{\SO(V_i),\epsilon_{i,1}}$. Under the isomorphism $$K^{\SO(V_i)}(\Gr_{\SO(V_i),\epsilon_{i,1}})\simeq \mathbb{C}[T_i]^{W_{\epsilon_{i,1}}},$$ 
the line bundle $\mathcal{L}_i$ corresponds to $x_{i,1}$. The torus fixed points in $\Gr_{\SO(V_i),\epsilon_{i,1}}$ are $\{z^{\epsilon_{i,r}}\mid r\in I_i\}$.

\begin{prop}\label{prop:orthmon}
Let $f\in \mathbb{C}[x^{\pm 1}]$ be a Laurent polynomial in one variable. Then

\begin{align*}
    z^*\circ(\iota_*)^{-1}&\big(f(\mathcal{L}_i) \cdot[\calO_{\mathcal{R}_{\epsilon_{i,1}}}]\big)
    =\sum_{r\in I_i}f(x_{i,r})\prod_{h\in Q_1\atop s(h)=i}\prod_{b=1}^{v_{t(h)}}(1-qx_{i,r}x_{t(h),b}^{-1})\\
    &\cdot \prod_{a=1}^{m_i}(1-qx_{i,r}w_{i,a}^{-1})\cdot \frac{1}{\prod_{\substack{1\leq s\leq v_i\\ s\neq r,r'}}(1-x_{i,s}x_{i,r}^{-1})}
    D_{i,r}.
\end{align*}
\end{prop}
\begin{proof}
    In this proof, $r\in I_i$. By \eqref{equ:wts}, 
    \[\Eu(T_{\epsilon_{i,r}}\Gr_{\SO(V_i),\epsilon_{i,1}})=\prod_{\substack{1\leq s\leq v_i\\ s\neq r,r'}}(1-x_{i,s}x_{i,r}^{-1}).\]
    Recall $\bfN$ from \eqref{eq:signedmatter}.
    By \eqref{equ:EuN}, we get
    \begin{align*}
        \Eu\bigg(z^{\epsilon_{i,r}}\bfN(\calO)/(z^{\epsilon_{i,r}}\bfN(\calO)\cap \bfN(\calO))\bigg)
        =\prod_{h\in Q_1\atop s(h)=i}\prod_{b=1}^{v_{t(h)}}(1-qx_{i,r}x_{t(h),b}^{-1})\prod_{a=1}^{m_i}(1-qx_{i,r}w_{i,a}^{-1}),
    \end{align*}
    Substituting these into \eqref{eq:minu-mono-op}, we get the desired formula. 
\end{proof}

\subsubsection{Symplectic case}
In this section, $i\in Q_S$.
The first coordinate cocharacter $\epsilon_{i,1}$ is quasi-minuscule for $\Sp(V_i)$. 
Let $\mathcal{L}_i$ denote
the tautological line bundle on $\Gr_{\Sp(V_i),\epsilon_{i,1}}$. Under the isomorphism $$K^{\Sp(V_i)}(\Gr_{\Sp(V_i),\epsilon_{i,1}})\simeq \mathbb{C}[T_i]^{W_{\epsilon_{i,1}}},$$ 
the line bundle $\mathcal{L}_i$ corresponds to $x_{i,1}$. The torus fixed points in $\Gr_{\Sp(V_i),\epsilon_{i,1}}$ are denoted by $z^{\epsilon_{i,r}}$ for $r\in I_i$. We will use the notations from Theorem \ref{thm:quasi}.
\begin{prop}\label{prop:sympmon}
Let $f\in \mathbb{C}[x^{\pm 1}]$ be a Laurent polynomial in one variable. Then
\begin{align*}
    z^*\circ (\iota_*)^{-1}&\bigg(\tilde{p}_*\big(f(\mathcal{L}_i)\cdot [\calO_{\tilde{\calR}_{G,\bfN,\leq \epsilon_{i,1}}}]\big)\bigg)=\sum_{r=1}^{v_i}f(x_{i,r})\prod_{h\in Q_1\atop s(h)=i}\prod_{a=1}^{v_{t(h)}}(1-qx_{i,r}x_{t(h),a}^{-1})\\
    &\cdot\prod_{a=1}^{m_i}(1-qx_{i,r}w_{i,a}^{-1})\cdot \frac{1}{(1-q^{-2}x_{i,r}^{-2})\prod_{1\leq s\neq r\leq v_i}(1-x_{i,s}x_{i,r}^{-1})}
    (D_{i,r}-1)+\cdots,\\
\end{align*}
where $\cdots$ denotes some polynomial in $K^{G\times F\times \tilde{\mathbb{C}}^\times_q}(\pt)$.
\end{prop}
\begin{proof}
    As in the orthogonal case, this follows from
    \[\Eu(T_{\epsilon_{i,r}}\Gr_{\Sp(V_i),\epsilon_{i,1}})
        =(1-q^{-2}x_{i,r}^{-2})\prod_{1\leq s\neq r\leq v_i}(1-x_{i,s}x_{i,r}^{-1}),\]
        and
    \begin{align*}
        \Eu\bigg(z^{\epsilon_{i,r}}\bfN(\calO)/(z^{\epsilon_{i,r}}\bfN(\calO)\cap \bfN(\calO))\bigg)
        =\prod_{h\in Q_1\atop s(h)=i}\prod_{a=1}^{v_{t(h)}}(1-qx_{i,r}x_{t(h),a}^{-1})\prod_{a=1}^{m_i}(1-qx_{i,r}w_{i,a}^{-1}).
    \end{align*}
\end{proof}

\begin{lemma}\label{lem:sympmono}
    Let 
    \[g(u):=\prod_{h\in Q_1\atop s(h)=i}\prod_{a=1}^{v_{t(h)}}(1-qux_{t(h),a}^{-1})\prod_{a=1}^{m_i}(1-quw_{i,a}^{-1}).\] Then for any $f\in \mathbb{C}[x^{\pm 1}]$,
    \begin{align*}
    z^*\circ (\iota_*)^{-1}&\bigg(\tilde{p}_*\big(f(\mathcal{L}_i)\cdot [\calO_{\tilde{\calR}_{G,\bfN,\leq \epsilon_{i,1}}}]\big)\bigg)=\sum_{r=1}^{v_i}\frac{f(x_{i,r})g(x_{i,r})}{(1-q^{-2}x_{i,r}^{-2})\prod_{1\leq s\neq r\leq v_i}(1-x_{i,s}x_{i,r}^{-1})}
    D_{i,r}\\
    &+\frac{1}{2}\frac{f(q^{-1})g(q^{-1})}{\prod_{r=1}^{v_i}(1-qx_{i,r})}+\frac{1}{2}\frac{f(-q^{-1})g(-q^{-1})}{\prod_{r=1}^{v_i}(1+qx_{i,r})}+\cdots\in \mathcal{A}_q,
\end{align*}
where $\cdots$ is some element in $K^{G\times F\times \tilde{\mathbb{C}}^\times_q}(\pt)$.
\end{lemma}
\begin{proof}
First of all,
\begin{align*}
    &\sum_{r=1}^{v_i}\frac{f(x_{i,r})g(x_{i,r})}{(1-q^{-2}x_{i,r}^{-2})\prod_{1\leq s\neq r\leq v_i}(1-x_{i,s}x_{i,r}^{-1})}\\
    =&\frac{1}{2}\sum_{r=1}^{v_i}\frac{f(x_{i,r})g(x_{i,r})-f(q^{-1})g(q^{-1})}{(1-q^{-1}x_{i,r}^{-1})\prod_{1\leq s\neq r\leq v_i}(1-x_{i,s}x_{i,r}^{-1})}+\frac{1}{2}\sum_{r=1}^{v_i}\frac{f(x_{i,r})g(x_{i,r})-f(-q^{-1})g(-q^{-1})}{(1+q^{-1}x_{i,r}^{-1})\prod_{1\leq s\neq r\leq v_i}(1-x_{i,s}x_{i,r}^{-1})}\\
    &+\frac{1}{2}\sum_{r=1}^{v_i}\frac{f(q^{-1})g(q^{-1})}{(1-q^{-1}x_{i,r}^{-1})\prod_{1\leq s\neq r\leq v_i}(1-x_{i,s}x_{i,r}^{-1})}+\frac{1}{2}\sum_{r=1}^{v_i}\frac{f(-q^{-1})g(-q^{-1})}{(1+q^{-1}x_{i,r}^{-1})\prod_{1\leq s\neq r\leq v_i}(1-x_{i,s}x_{i,r}^{-1})}.
\end{align*}
By the localization theorem on $\Sp(V_i)/P_{\epsilon_{i,1}}$, the sum of the first two summands lies in $K^{G\times F\times \tilde{\mathbb{C}}^\times_q}(\pt)$.
Moreover, applying the residue theorem to 
\[\frac{z^{v_i}}{(z-q^{-1})\prod_{1\leq  r\leq v_i}(z-x_{i,r})}dz,\]
we get
\[\sum_{r=1}^{v_i}\frac{1}{(1-q^{-1}x_{i,r}^{-1})\prod_{1\leq s\neq r\leq v_i}(1-x_{i,s}x_{i,r}^{-1})}=1-\frac{1}{\prod_{r=1}^{v_i}(1-qx_{i,r})}.\]
Since $f(q^{-1})g(q^{-1}), f(-q^{-1})g(-q^{-1})\in K^{G\times F\times \tilde{\mathbb{C}}^\times_q}(\pt)$, we get by Proposition \ref{prop:sympmon} that 
\begin{align*}
    z^*\circ (\iota_*)^{-1}&\bigg(\tilde{p}_*\big(f(\mathcal{L}_i)\cdot [\calO_{\tilde{\calR}_{G,\bfN,\leq \epsilon_{i,1}}}]\big)\bigg)=\sum_{r=1}^{v_i}\frac{f(x_{i,r})g(x_{i,r})}{(1-q^{-2}x_{i,r}^{-2})\prod_{1\leq s\neq r\leq v_i}(1-x_{i,s}x_{i,r}^{-1})}
    D_{i,r}\\
    &+\frac{1}{2}\frac{f(q^{-1})g(q^{-1})}{\prod_{r=1}^{v_i}(1-qx_{i,r})}+\frac{1}{2}\frac{f(-q^{-1})g(-q^{-1})}{\prod_{r=1}^{v_i}(1+qx_{i,r})}+\cdots\in \mathcal{A}_q,
\end{align*}
where $\cdots$ is some element in $K^{G\times F\times \tilde{\mathbb{C}}^\times_q}(\pt)$. This finishes the proof.
\end{proof}

\subsection{GKLO operators}
We need to enlarge the coefficient ring. 
For $i\in Q_0^+$ let
\[
 \widetilde{\GL(V_i)}=\{(g,y):\det g=y^2\},\qquad
 \widetilde G=\prod_{i\in Q_0^+}\widetilde{\GL(V_i)}
 \times\prod_{i\in Q_O}\SO(V_i)\times\prod_{i\in Q_S}\Sp(V_i).
\]
We use $\widetilde{\calD_q(T)}$ to denote the algebra obtained from $\calD_q(T)$ by adjoining
$\{x_{i,r}^{\pm 1/2}=x_{\tau i,r}^{\mp 1/2}\mid i\notin Q_0^\tau, 1\leq r\leq v_i\}$, which obey the rule
$D_{i,r}x_{i,r}^{1/2}=qx_{i,r}^{1/2}D_{i,r}$, and put
\[\widetilde{\calD_q(T)}_{\loc}=\widetilde{\calD_q(T)}\otimes_{\mathbb C[q^{\pm1}]}\mathbb C(q).\]
Define $\widetilde{\mathcal A_q}(G\times F, \bfN)$ to be the subalgebra in 
$K^F(\pt)\otimes_{\mathbb C}\widetilde{\calD_q(T)}$
generated by
$z^*\circ(\iota_*)^{-1}(\mathcal A_q(G\times F, \bfN))$ and $K^{\widetilde G\times F}(\pt)$.

For $i\in Q_0^+$, set
\[
M_i=\prod_{r=1}^{v_i}x_{i,r}^{1/2}
\in K^{\widetilde G\times F}(\pt).
\]
For the remaining vertices, define
\[
M_{\tau i}=M_i^{-1}\quad(i\in Q_0^+),
\qquad
M_i=1\quad(i\in Q_0^\tau).
\]
Thus $M_{\tau i}=M_i^{-1}$ for every $i\in Q_0$.
Define
\begin{align*}
 V_i(z)&=M_i^{-1}\prod_{r=1}^{v_i}(1-x_{i,r}/z),&
 \hat V_i(z)&=M_i\prod_{r=1}^{v_i}(1-z/x_{i,r}),\\
 V_{i,r}(z)&=M_i^{-1}\prod_{s\ne r}(1-x_{i,s}/z),&
 \hat V_{i,r}(z)&=M_i\prod_{s\ne r}(1-z/x_{i,s}).
\end{align*}
Also set
\begin{equation*}
 W_i(z)=\prod_{b=1}^{m_i}(1-qz/w_{i,b}),\qquad
 A_i(z)=W_i(z)\prod_{h\in Q_1,s(h)=i}\hat V_{t(h)}(qz).
\end{equation*}
Put $e_i=\mathbf1_{\{i\to\tau i\}}$, which is equal to $1$ if there is an edge $h$ going from $i$ to $\tau i$ and $0$ otherwise;
it is zero at a fixed vertex.

\begin{definition}
At a nonfixed vertex let
\begin{equation*}
 \hat B_i(u)=\sum_{r=1}^{v_i}\delta(qx_{i,r}u)
 \frac{A_i(x_{i,r})}{(1-qx_{i,r}^2)^{e_i}V_{i,r}(x_{i,r})}D_{i,r}.
\end{equation*}
At an orthogonal fixed vertex let
\begin{equation}\label{eq:Borthogonal}
 \hat B_i(u)=\sum_{r\in I_i}\delta(qx_{i,r}u)
 \frac{(1-x_{i,r}^{-2})A_i(x_{i,r})}{V_{i,r}(x_{i,r})}D_{i,r}.
\end{equation}
At a symplectic fixed vertex let
\begin{equation}\label{eq:Bsymplectic}
\begin{split}
 \hat B_i(u)={}&\sum_{r=1}^{v_i}\delta(qx_{i,r}u)
 \frac{A_i(x_{i,r})}{(1-q^{-2}x_{i,r}^{-2})V_{i,r}(x_{i,r})}D_{i,r}\\
 &+\sum_{\epsilon=\pm1}\frac{\delta(\epsilon u)A_i(\epsilon q^{-1})}
 {2V_i(\epsilon q^{-1})}.
\end{split}
\end{equation}
For every vertex put
\begin{equation*}
 \hat\Theta_i(u)=(q-q^{-1})^2p_i(u)
 \frac{A_i(q^{-1}u^{-1})A_{\tau i}(q^{-1}u)}
 {V_i(q^{-1}u^{-1})V_{\tau i}(q^{-1}u)},
\end{equation*}
where
\begin{equation*}
 p_i(u)=\begin{cases}
 (qu)^{e_i}(-qu^{-1})^{e_{\tau i}}(u-u^{-1})^{c_{i,\tau i}},
 &i\notin Q_0^\tau,\\
 -q(1-q^2u^2)(1-q^2u^{-2}),&i\in Q_O,\\
 q/(u-u^{-1})^2,&i\in Q_S.
 \end{cases}
\end{equation*}
Here we expand the rational function $\hat \Theta_i(u)$ at $u=\infty$.
\end{definition}

Define 
\begin{equation}\label{eq:muuniform}
\ell_i:=m_i-v_i+\sum_{h:s(h)=i}v_{t(h)}+d_i,
 \qquad
 d_i:=\begin{cases}-2e_i,&i\notin Q_0^\tau,\\
 2,&i\in Q_O,\\-2,&i\in Q_S.\end{cases}
\end{equation}
Direct computations give the following lemma.
\begin{lemma}\label{lem:Thetauniform}
As rational functions,
$\hat\Theta_i(u)=\hat\Theta_{\tau i}(u^{-1})$. At infinity,
\[
 \hat\Theta_i(u)=k_{\tau i}u^{\ell_{\tau i}}
 +\text{terms of smaller degree in }u,
\]
where $k_{\tau i}$ is a unit in
$K^{\widetilde G\times F}(\pt)\otimes_{\mathbb{C}}\mathbb C(q)$.
For a nonfixed vertex it is given by
\begin{equation*}
 k_{\tau i}=q^{e_i}(-q)^{e_{\tau i}}(q-q^{-1})^2
 \frac{\prod_{b=1}^{m_{\tau i}}(-w_{\tau i,b})^{-1}}
 {\prod_{r=1}^{v_i}(-qx_{i,r})}
 \frac{\prod_{h:s(h)=i}M_{t(h)}}{\prod_{h:s(h)=\tau i}M_{t(h)}}
 \prod_{h:s(h)=\tau i}(-1)^{v_{t(h)}}.
\end{equation*}
At a fixed vertex it is
\begin{equation*}
 k_i=
 (q-q^{-1})^2(-q)^{-v_i}
 \prod_{b=1}^{m_i}(-w_{i,b})^{-1}
 \prod_{h:s(h)=i}(-1)^{v_{t(h)}}\cdot\begin{cases}q^3,&i\in Q_O,\\q,&i\in Q_S\end{cases}.
\end{equation*}
All coefficients of $\hat\Theta_i(u)$ at infinity belong to the
same equivariant coefficient ring.
\end{lemma}

Define $\hat B_{i,s}$ and $\hat\Theta_{i,r}$ by
\[
 \hat B_i(u)=\sum_s\hat B_{i,s}u^{-s},\qquad
 \hat\Theta_i(u)=k_{\tau i}\left(u^{\ell_{\tau i}}+
 (q-q^{-1})\sum_{r>-\ell_{\tau i}}\hat\Theta_{i,r}u^{-r}\right).
\]

\begin{prop}\label{prop:ThetaBS}
The elements $(q-q^{-1})\hat\Theta_{i,r}$ and $k_i$ belong to
$\widetilde{\mathcal A_q}$. Moreover,
\begin{equation*}
 \hat B_{i,s}=\begin{cases}
 q^{-s}M_i\prod_{h:s(h)=i}M_{t(h)}\cdot z^*\circ(\iota_*)^{-1}\bigl(\mathcal Q_i^{-s}\cdot [\calO_{\calR_{\epsilon_{i,1}}}]\bigr),
 &i\in Q_0^+,\\
 q^sM_i\prod_{h:s(h)=i}M_{t(h)}\cdot z^*\circ(\iota_*)^{-1}\bigl(\mathcal S_{\tau i}^{s}\cdot [\calO_{\calR_{-\epsilon_{\tau i,v_i}}}]\bigr),
 &i\in\tau(Q_0^+),\\
 q^{-s}\prod_{h:s(h)=i}M_{t(h)}\cdot z^*\circ(\iota_*)^{-1}\bigl(\mathcal L_i^{-s}\cdot [\calO_{\calR_{\epsilon_{i,1}}}]\bigr),
 &i\in Q_O,\\
 q^{-s}\prod_{h:s(h)=i}M_{t(h)}\cdot z^*\circ (\iota_*)^{-1}\bigg(\tilde{p}_*\big(\mathcal L_i^{-s}\cdot [\calO_{\tilde{\calR}_{G,\bfN,\leq \epsilon_{i,1}}}]\big)\bigg)+\cdots,&i\in Q_S,
 \end{cases}
\end{equation*}
where $\cdots$ denotes some element in $K^{\widetilde{G}\times F\times \widetilde{\mathbb{C}}_q^\times}(\pt)$.
\end{prop}
\begin{proof}
    The first statement follows from Lemma \ref{lem:Thetauniform}. The rest follows from Proposition \ref{prop:fmono}, \ref{prop:fmono2}, \ref{prop:orthmon}, and Lemma \ref{lem:sympmono}.
\end{proof}

\section{GKLO homomorphisms and Coulomb branch realizations}\label{sec:gklo}
In this section, we construct the GKLO-type homomorphism for the shifted affine $\mathrm{i}$quantum groups and prove that it factors through the corresponding quantized Coulomb branches. 

\subsection{Formal distributions and expansion conventions}\label{sec:formal-distributions}

Let $K$ be the field of rational functions in the coordinate variables and parameters. The spectral variables $u,v,\ldots$ are independent formal variables over $K$. A formal distribution is a sum of integral powers of the
spectral variables with coefficients in $K$.

Set $\delta(u)=\sum_{m\in\mathbb Z}u^m$. For $a,b\in K^\times$, $f\in K[u^{\pm1}]$, and $G\in K[u^{\pm1},v^{\pm1}]$, we have
\begin{equation}\label{eq:delta-laurent}
    \begin{aligned}
         f(u)\delta(a/u)&=f(a)\delta(a/u),\\
         G(u,v)\delta(a/u)\delta(b/v)
         &=G(a,b)\delta(a/u)\delta(b/v).
    \end{aligned}
\end{equation}
For any formal distribution $S$ in one variable, we also have
\[
    \delta(uv)S(v)=\delta(uv)S(u^{-1}).
\]

For $F(u)\in K(u)$, write
\[
    F^+(u)=\iota_\infty^u F(u)\in K((u^{-1})),
    \qquad
    F^-(u)=\iota_0^u F(u)\in K((u)),
\]
as its expansions at infinity and zero in $u$, respectively.

The following lemma explains the coefficient-wise delta substitutions used below.

\begin{lemma}\label{lem:coefficientwise-delta}
    Suppose
    \[
        A(z,w)=\sum_{k\ge k_0}p_k(w)z^{-k},\qquad S(z)=\sum_{r\ge r_0}s_rz^{-r},
    \]
    where $k_0,r_0\in\mathbb Z$, $p_k(w)\in K[w^{\pm1}]$, and $s_r\in K$. For $a\in K^\times$, the products below are well defined and
    \[
        A(z,w)S(z)\delta(a/w) = \left( \sum_{k\ge k_0}p_k(a)z^{-k} \right) S(z)\delta(a/w).
    \]
    The identity remains valid after multiplication by $\delta(zv)$, where $v$ is an additional independent spectral variable. The analogous statements hold for expansions at $z=0$.
\end{lemma}

\begin{proof}
    By \eqref{eq:delta-laurent}, the coefficient of $z^{-N}w^{-M}$ on either side is
    \[
        a^M\sum_{k=k_0}^{N-r_0}p_k(a)s_{N-k},
    \]
    with an empty sum interpreted as zero. This is a finite sum. After inserting $\delta(zv)$, fixing the power of $v$ fixes its summation index. The statements at zero follow by replacing $z$ with $z^{-1}$.
\end{proof}

The same finiteness argument shows that $A(z,w)S(z)B(w)$ is well defined for an arbitrary bilateral current $B(w)$. For the rational coefficients in the following $\Theta$--$B$ and $\mathrm{i}$-Serre relations, the prescribed expansions have Laurent polynomial coefficients in the remaining $B$-variable. Lemma \ref{lem:coefficientwise-delta} therefore applies to each delta summand of the represented $B$-current.

In difference-operator expressions, all shift operators are placed on the right before these identities are applied. Thus the delta parameters used for substitution are those in the normal-ordered expression.

For $\mathbf{a}=(a_1,\ldots,a_n)\in(K^\times)^n$, define
\begin{equation*}
    \Delta_{\mathbf{a}}(\mathbf{u})=\prod_{i=1}^n\delta\left( \frac{u_i}{a_i} \right) .
\end{equation*}
Then for any $f\in K[u_1^{\pm 1},\ldots,u_n^{\pm 1}]$, we have $f(\mathbf{u})\Delta_{\mathbf{a}}(\mathbf{u})=f(\mathbf{a})\Delta_{\mathbf{a}}(\mathbf{u})$.

We will use the following cancellation criterion for finite delta sums.
\begin{lemma}\label{lem:delta-cancellation}
    Suppose
    \begin{equation*}
        T(\mathbf{u}) = \sum_{k=1}^N\Delta_{\mathbf{a}_k}(\mathbf{u})A_k,
    \end{equation*}
    where each $A_k$ is a difference operator with all shifts placed on the right. If $P\in K[u_1^{\pm 1},\ldots,u_n^{\pm 1}]$ satisfies $P(\mathbf{a}_k)\neq 0$ for all $1\leq k\leq N$, then
    \begin{equation*}
        P\cdot T = 0 \implies T = 0.
    \end{equation*}
\end{lemma}

\begin{proof}
    By polynomial interpolation on the finite set of distinct support points, there exists a polynomial $Q\in K[\mathbf{u}]$ such that $Q(\mathbf{a}_k)=\frac{1}{P(\mathbf{a}_k)}$ for all $1\leq k\leq N$. Hence the product $Q\cdot P\cdot T$ is well-defined, and
    \begin{equation*}
        0=Q\cdot P\cdot T=\sum_{k=1}^NQ(\mathbf{a}_k)P(\mathbf{a}_k)\Delta_{\mathbf{a}_k}(\mathbf{u})A_k=\sum_{k=1}^N\Delta_{\mathbf{a}_k}(\mathbf{u})A_k=T.
    \end{equation*}
\end{proof}

\subsection{Shifted affine $\mathrm{i}$quantum groups}\label{sec:iquantum}
Let $\mathbb I=Q_0$, and let $C=(c_{ij})$ be the Cartan matrix of the
underlying graph: $c_{ii}=2$, $c_{ij}=-1$ for adjacent vertices and
$c_{ij}=0$ otherwise. The involution preserves $C$. 

Let $\Lambda_i$ be the fundamental coweights and choose an integral
coweight $\mu$, with $\ell_i=\langle\mu,\alpha_i\rangle$.
Put $[n]=(q^n-q^{-n})/(q-q^{-1})$.
Following \cite[Definition 2.4]{LWW26}, we give the following definition of the affine $\mathrm{i}$quantum group $\Ui_\mu$.
\begin{definition}\label{def:iquant}
The algebra $\Ui_\mu$ is generated over $\mathbb C(q)$ by
$\mathbb K_i^{\pm1}$, $B_{i,s}$ ($s\in\mathbb Z$), and
$\Theta_{i,r}$ ($r>-\ell_{\tau i}$). Write
\begin{equation*}
 \Theta_i(u)=\mathbb K_{\tau i}\left(u^{\ell_{\tau i}}+
 (q-q^{-1})\sum_{r>-\ell_{\tau i}}\Theta_{i,r}u^{-r}\right),
 \qquad B_i(u)=\sum_{s\in\mathbb Z}B_{i,s}u^{-s}.
\end{equation*}
The defining current relations are
\begin{align}
 &[\Theta_i(u),\Theta_j(v)]=0,\label{rel:HH}\\
 &\Theta_i(u)B_j(v)=
 \frac{(q^{c_{ij}}u-v)(q^{c_{\tau i,j}}u^{-1}-v)}
 {(u-q^{c_{ij}}v)(u^{-1}-q^{c_{\tau i,j}}v)}B_j(v)\Theta_i(u),\label{rel:HB}\\
 &[B_i(u),B_{\tau i}(v)]=\frac{\delta(uv)}{q-q^{-1}}
 (\Theta_i(u)-\Theta_{\tau i}(v)),\quad c_{i,\tau i}=0,\label{rel:BB1}\\
 &(u-q^2v)B_i(u)B_i(v)+(v-q^2u)B_i(v)B_i(u)\notag\\
 &\hspace{15mm}=\frac{\delta(uv)(u-v)}{q^{-1}-q}
 (\Theta_i(u)-\Theta_i(v)),\quad i=\tau i,\label{rel:BB2}\\
 &(u-q^{-1}v)B_i(u)B_{\tau i}(v)
 +(v-q^{-1}u)B_{\tau i}(v)B_i(u)\notag\\
 &\hspace{15mm}=\frac{\delta(uv)(u-v)}{q^2-1}
 (\Theta_i(u)-\Theta_{\tau i}(v)),\quad c_{i,\tau i}=-1,\label{rel:BB3}\\
 &[B_i(u),B_j(v)]=0,\quad c_{ij}=0,\ j\ne\tau i,\label{rel:BB4}\\
 &(u-q^{c_{ij}}v)B_i(u)B_j(v)
 +(v-q^{c_{ij}}u)B_j(v)B_i(u)=0,\quad j\ne\tau i.\label{rel:BB5}
\end{align}
Define
\[
 \begin{aligned}
 \mathbb S_{i,j}(u_1,u_2\mid v)=\Sym_{u_1,u_2}\bigl(&B_i(u_1)B_i(u_2)B_j(v)\\
 &-[2]B_i(u_1)B_j(v)B_i(u_2)+B_j(v)B_i(u_1)B_i(u_2)\bigr).
 \end{aligned}
\]
The Serre relations are
\begin{align}
 &\mathbb S_{i,j}(u_1,u_2\mid v)=0,
 \quad c_{ij}=-1,\ i\ne\tau i,\ j\ne\tau i,\label{rel:Serre1}\\
 &\mathbb S_{i,j}(u_1,u_2\mid v)=
 \frac{\delta(u_1u_2)(u_1-u_2)v}{(u_1-qv)(u_2-qv)}
 (\Theta_i(u_1)-\Theta_i(u_2))B_j(v),\quad c_{ij}=-1,\ i=\tau i,\label{rel:Serre2}\\
 &\mathbb S_{i,\tau i}(u_1,u_2\mid v)=\Sym_{u_1,u_2}
 \frac{\delta(u_2v)(1+q^2)(u_2-v)v}{(qu_1-v)(v-q^2u_1^{-1})}
 (\Theta_i(u_2)-\Theta_{\tau i}(v))B_i(u_1),\notag\\
 &\hspace{85mm}c_{i,\tau i}=-1.\label{rel:Serre3}
\end{align}
\end{definition}

\begin{remark}\label{rem:expansions}
We use the conventions of Section \ref{sec:formal-distributions}.
The rational coefficient in \eqref{rel:HB} is expanded at $u=\infty$. Put
\begin{align*}
    h_v^{(2)}(u)&:=\frac{(u-u^{-1})v}{(u-qv)(u^{-1}-qv)},\\
    h_{u_1}^{(3)}(u)&:=\frac{(1+q^2)(u-u^{-1})u^{-1}}{(qu_1-u^{-1})(u^{-1}-q^2u_1^{-1})}.
\end{align*}
The right-hand sides of \eqref{rel:Serre2} and
\eqref{rel:Serre3} mean, respectively,
\begin{align*}
    &\delta(u_1u_2)
    \bigl(h_v^{(2),+}(u_1)\Theta_i(u_1)-h_v^{(2),-}(u_1)\Theta_i(u_1^{-1})\bigr)B_j(v),\\
    &\operatorname{Sym}_{u_1,u_2}\delta(u_2v)
    \bigl(h_{u_1}^{(3),+}(u_2)\Theta_i(u_2)-h_{u_1}^{(3),-}(u_2)\Theta_{\tau i}(u_2^{-1})\bigr)B_i(u_1)
\end{align*}
In a representation where these two $\Theta$-series come from the same rational function $\hat{\Theta}_i(u)$, put
\[
    \mathcal H_{i,a}^{(r)}(u)
    :=h_a^{(r)}(u)\hat{\Theta}_i(u),
    \qquad r=2,3.
\]
The parenthesized terms above are the expansion differences of $\mathcal H_{i,v}^{(2)}(u_1)$ and $\mathcal H_{i,u_1}^{(3)}(u_2)$, respectively. Here $+$ and $-$ denote expansion of the complete rational product at infinity and zero in its argument. These interpretations of \eqref{rel:Serre2} and \eqref{rel:Serre3} are consistent with the calculations in \cite[Sections~4.5 and~5.2.2]{LWW26}, respectively.
\end{remark}

\subsection{The homomorphism and its geometric factorization}
Now we can state the main result of this paper. Let $\mu=\sum_i\ell_i\Lambda_i$, where $\ell_i$ is defined in \eqref{eq:muuniform}.
\begin{theorem}\label{thm:gklo}
For the signed-quiver data in Section \ref{sec:quiver},
the assignments
\[
 \mathbb K_i\longmapsto k_i,\qquad
 \Theta_i(u)\longmapsto\hat\Theta_i(u)^+,\qquad
 B_i(u)\longmapsto\hat B_i(u)
\]
define a unique $K^F(\pt)\otimes_{\mathbb{C}}\mathbb C(q)$-algebra homomorphism
\[
 \psi:\Ui_\mu\otimes_{\mathbb C}K^F(\pt)
 \longrightarrow K^F(\pt)\otimes_{\mathbb C}
 \widetilde{\calD_q(T)}_{\loc}.
\]
\end{theorem}
\begin{proof}
The leading terms have the prescribed degrees by
Lemma~\ref{lem:Thetauniform}. Appendices~\ref{sec:appen1}--\ref{sec:appen3}
verify the full presentation of Definition~\ref{def:iquant}.
\end{proof}
\begin{remark}
    If we change to $\eta_h=1$ for all $h\in Q_1^\tau$ in \eqref{eq:arrowsigns}, the second wedge square in \eqref{eq:signedmatter} will become the second symmetric square. We can get similar results. For the cohomology case, see \cite{wang2025quivers}.
\end{remark}
Recall that $M_i=\prod_{r=1}^{v_i}x_{i,r}^{1/2}$ for $i\notin Q_0^{\tau}$ and $M_i=1$ for $i\in Q_0^{\tau}$.
\begin{theorem}\label{thm:Coulomb}
   The homomorphism $\psi$ in Theorem \ref{thm:gklo} factors through a unique $K^F(\pt)\otimes_{\mathbb{C}}\mathbb{C}(q)$-algebra homomorphism $\Psi$:
   \[\xymatrix{\Ui_\mu\otimes_{\mathbb{C}} K^F(\pt) \ar[rr]^-\Psi \ar[rd]^-\psi & &\widetilde{\mathcal{A}_q}\otimes_{\mathbb{C}[q^{\pm 1}]}\mathbb{C}(q) \ar@{^{(}->}[ld]^-{z^*\circ(\iota_*)^{-1}}\\
   &K^F(\pt)\otimes_{\mathbb{C}}\widetilde{\calD_q(T)}_{\loc},}\] 
    such that for any $i\in Q_0$, $\Psi(\mathbb{K}_{i})=k_i$, $\Psi(\Theta_{i,r})=\hat{\Theta}_{i,r}$; and
    \[\Psi(B_{i,s})= \begin{cases}
    q^{-s}M_i\prod_{h:s(h)=i}M_{t(h)}\cdot \mathcal Q_i^{-s}\cdot [\calO_{\calR_{\epsilon_{i,1}}}],&i\in Q_0^+,\\
    q^sM_i\prod_{h:s(h)=i}M_{t(h)}\cdot \mathcal S_{\tau i}^{s}\cdot [\calO_{\calR_{-\epsilon_{\tau i,v_i}}}],
    &i\in\tau(Q_0^+),\\
    q^{-s}\prod_{h:s(h)=i}M_{t(h)}\cdot \mathcal L_i^{-s}\cdot [\calO_{\calR_{\epsilon_{i,1}}}],&i\in Q_O,\\
    q^{-s}\prod_{h:s(h)=i}M_{t(h)}\cdot \tilde{p}_*\big(\mathcal L_i^{-s}\cdot [\calO_{\tilde{\calR}_{G,\bfN,\leq \epsilon_{i,1}}}]\big)+\cdots,&i\in Q_S,
 \end{cases}\]
where $\cdots$ denotes some element in $K^{\widetilde{G}\times F\times \widetilde{\mathbb{C}}_q^\times}(\pt)$ determined by Proposition \ref{prop:ThetaBS}.
\end{theorem}
\begin{proof}
This follows from Theorem \ref{thm:gklo} and Proposition \ref{prop:ThetaBS}.
\end{proof}

\begin{remark}
The operators are obtained from the signed gauge theory, so their
geometric interpretation is part of the construction. GKLO
homomorphisms originated in \cite{GKLO05,GKLOquant} and have been
extended to shifted Yangians and quantum affine algebras in
\cite{KWWY14,BFN19,FT19}. For the $\mathrm{i}$analogues see
\cite{bartlett2025gklo,li2026gklorepresentationsshiftedquantum,lu2025shifted,LWW26,SSX26,SSX26b,wang2025quivers}. Our operators $\hat B_i(u)$ are quite different than the ones $\hat B^{\mathrm{LWW}}_i(u)$ in \cite{LWW26}. The symplectic current \eqref{eq:Bsymplectic} contains two scalar delta terms, supported at $u=\pm1$. By comparison, the inhomogeneous fixed-vertex current in \cite[Theorem~3.1, (3.5)]{LWW26} contains a single scalar delta term, supported at $u=1$.
\end{remark}

\subsection{Specializations of the construction}
This construction gives a uniform treatment for the Satake diagram $\mathrm{AIII}_{n}^{(\tau)}$.

\subsubsection{Fixed-point-free involutions}
If $Q_0^\tau=\varnothing$, then
\[
 G=\prod_{i\in Q_0^+}\GL(V_i),\qquad
 \bfN\cong E_V^\sigma\oplus\bigoplus_{i\in Q_0}\Hom(V_i,W_i),
\]
where $E_V^\sigma$ is the signed fixed arrow space. Taking all arrow
signs to be $-1$ gives the convention
$f_{\tau h}=-f_h^t$ and the exterior-square term at each fixed arrow.

\begin{example}[Type $\mathrm{AIII}_{2n}^{(\tau)}$]
Take the Satake diagram in Figure \ref{fig:AIII2n}, with $\tau(i)=2n+1-i$ and
$Q_0^+=\{1,\ldots,n\}$. All vertices are nonfixed. Orient the
underlying quiver as
\[
 1\longrightarrow 2\longrightarrow\cdots\longrightarrow n
 \longrightarrow n+1\longrightarrow\cdots\longrightarrow 2n,
\]
so that $\tau$ reverses arrows, and choose all arrow signs to be $-1$.
The arrows $i\to i+1$, $1\leq i<n$, represent the nonfixed arrow
orbits; the middle arrow $n\to n+1$ is fixed.
\begin{figure}[htbp]
\centering
\begin{tikzpicture}[scale=.4]
\foreach \x/\lab/\dual in {0/1/{2n},4/2/{2n-1},12/{n-1}/{n+2},16/n/{n+1}} {
 \draw[thick] (\x,0) circle (.3);
 \draw[thick] (\x,-3) circle (.3);
 \node (evenTop\x) at (\x,.75) {$\lab$};
 \node (evenBottom\x) at (\x,-3.75) {$\dual$};
}
\node at (8,0) {$\dots$};
\node at (8,-3) {$\dots$};
\foreach \y in {0,-3} {
 \draw[thick] (.3,\y) -- (3.7,\y);
 \draw[thick] (4.3,\y) -- (6.8,\y);
 \draw[thick] (9.2,\y) -- (11.7,\y);
 \draw[thick] (12.3,\y) -- (15.7,\y);
}
\draw[thick] (16,-.3) -- (16,-2.7);
\foreach \x in {-.5,3.5,15.5}
 \draw[thick,<->,red,bend left=50]
 (\x+.3,-2.5) to (\x+.3,-.5);
\draw[thick,<->,red,bend left=50] (11.7,-2.5) to (11.7,-.5);
\end{tikzpicture}
\caption{Satake diagram $\mathrm{AIII}_{2n}^{(\tau)}$}
\label{fig:AIII2n}
\end{figure}
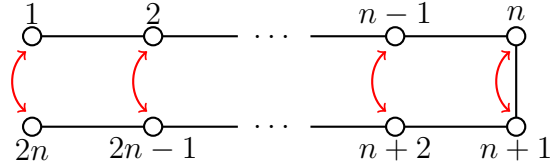
We obtain
\begin{align*}
 G_V^\sigma\cong\prod_{i=1}^n\GL(V_i),\quad
 E_V^\sigma\cong\bigoplus_{i=1}^{n-1}\Hom(V_i,V_{i+1})
                 \oplus\bigwedge\nolimits^2 V_n^*.
\end{align*}
Indeed, a map $f:V_n\to V_{n+1}=V_n^*$ at the fixed middle
arrow satisfies $f=-f^t$, so it is an alternating bilinear form on
$V_n$. The displayed exterior-square summand carries the
contragredient action of $\GL(V_n)$. This example is discussed in \cite[Section 3(ix)(C)]{Nak25}.
\end{example}

\subsubsection{Type $\mathrm{AIII}_{2n-1}^{(\tau)}$}
For the diagram $A_{2n-1}$ take $\tau(i)=2n-i$, $Q_0^+=\{1,\ldots,n-1\}$,
and $Q_0^\tau=\{n\}$. Choose any orientation reversed by
$\tau$. At the fixed vertex, choose
$W_n=U_n\oplus U_n^*$. The common polarization is
\begin{equation*}
 \bfN=\bigoplus_{h\in Q_1^+}\Hom(V_{s(h)},V_{t(h)})
 \oplus\bigoplus_{\substack{1\leq i\leq 2n-1\\ i\neq n}}\Hom(V_i,W_i)\oplus U_n\otimes V_n,
\end{equation*}
with flavor group
$F=\displaystyle\prod_{\substack{1\leq i\leq 2n-1\\ i\neq n}}\GL(W_i)\times\GL(U_n)$.

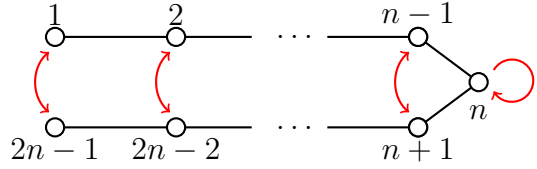
\begin{figure}[htbp]
\centering
\begin{tikzpicture}[scale=.4]
\node at (0,0.75) {$1$};
\node at (4,0.75) {$2$};
\node at (12,0.75) {$n-1$};    
\node at (0,-3.75) {$2n-1$};
\node at (4,-3.75) {$2n-2$};
\node at (12,-3.75) {$n+1$};
\node at (8,0) {$\dots$};
\node at (8,-3) {$\dots$};

\draw[thick] (12,0) circle (0.3);
\draw[thick] (12,-3) circle (0.3);

\draw[thick] (0,0) -- (6.5,0);
\draw[thick] (0,-3) -- (6.5,-3);
    
\draw[thick] (9,0) -- (12, 0) -- (14,-1.5);
\draw[thick] (9,-3) -- (12, -3) -- (14,-1.5);

    \foreach \x in {0,2,6} 
{\draw[thick,xshift=\x cm,fill=white] (\x, 0) circle (0.3); 
\draw[thick,xshift=\x cm,fill=white] (\x, -3) circle (0.3);}
\draw[thick,fill=white] (14,-1.5) circle (0.3);
\node at (14,-2.5) {$n$};

\draw[red, thick, <-] (14.5, -1.8) arc[start angle=210, end angle=510, radius=0.7];

\foreach \x in {-.5,3.5}
\draw[thick,<->, red, bend left=50] (\x+0.3,-2.5) to (\x+0.3,-0.5);
\draw[thick,<->, red, bend left=50] (11.7,-2.5) to (11.7,-0.5);

\end{tikzpicture}
\caption{Satake diagram $\mathrm{AIII}_{2n-1}^{(\tau)}$}

\end{figure}

\begin{example}
Take $Q_O=\{n\}$, with $v_n\geq3$, and $Q_S=\varnothing$.
Then $V_n$ is orthogonal, $W_n$ is symplectic, and
\[
 G=\prod_{i=1}^{n-1}\GL(V_i)\times\SO(V_n).
\]
The fixed current is \eqref{eq:Borthogonal}; its modes are the
minuscule classes dressed by $\mathcal L_n^{-s}$. Theorem~\ref{thm:Coulomb} gives the orthogonal type $\mathrm{AIII}_{2n-1}^{(\tau)}$ realization.
\end{example}

\begin{example}
Take $Q_S=\{n\}$ and $Q_O=\varnothing$, with $v_n$ positive and even.
Then $V_n$ is symplectic, $W_n$ is orthogonal of even dimension, and
\[
 G=\prod_{i=1}^{n-1}\GL(V_i)\times\Sp(V_n).
\]
The fixed current is \eqref{eq:Bsymplectic}. Its two terms at $u=\pm1$
complete the quasi-minuscule monopole as in Lemma~\ref{lem:sympmono}.
Theorem~\ref{thm:Coulomb} gives the
symplectic type $\mathrm{AIII}_{2n-1}^{(\tau)}$ realization.
\end{example}

These polarizations give cotangent versions of the gauge theories
in \cite[Section 3(ix)(B)]{Nak25}.

\appendix

\section{Paired-vertex relations}\label{sec:appen1}

Throughout the appendices, we use the formal-distribution conventions of Section~\ref{sec:formal-distributions}.

We verify the relations in Definition~\ref{def:iquant} for the operators $\hat{B}_i(u)$ and $\hat{\Theta}_i(u)$, beginning with those involving nonfixed vertices.

\subsection{Exchange factors and residues}

The following formal-distribution identity converts
the difference of two expansions into delta functions. 
\begin{lemma}[{\cite[Lemma C.2]{FT19}}]\label{lem:res}
    For any rational function $ \gamma(u) $ with simple poles $ \{a_k\}\subset K^\times $ and possibly poles of higher order at $ u=0,\infty $, we have
    \begin{equation*}
        \gamma(u)^+-\gamma(u)^-=\sum_k\delta\left( \frac{u}{a_k} \right) \Res_{u=a_k}\frac{1}{u}\gamma(u)du,
    \end{equation*}
    where $ \gamma(u)^+ $ (resp. $ \gamma(u)^- $) is the Laurent expansion of $ \gamma(u) $ at $ \infty $ (resp. $ 0 $).
\end{lemma}

For a nonfixed vertex $i\notin Q_0^\tau$ put
\[
 X_{i,r}=\frac{A_i(x_{i,r})}
 {(1-qx_{i,r}^2)^{e_i}V_{i,r}(x_{i,r})}D_{i,r},
 \qquad \hat{B}_i(u)=\sum_r\delta(qx_{i,r}u)X_{i,r}.
\]
For a fixed vertex $j$, $X_{j,s}$ denotes the coefficient of
$\delta(qx_{j,s}v)$ in \eqref{eq:Borthogonal} or
\eqref{eq:Bsymplectic}. The symplectic case includes the two
indices $s=\epsilon\in\{+,-\}$, with
$x_{j,\epsilon}=\epsilon q^{-1}$ and
$X_{j,\epsilon}=A_j(\epsilon q^{-1})/(2V_j(\epsilon q^{-1}))$.
No difference operator is attached to these two indices.
For sums of current terms write $\mathcal J_j=I_j$ outside $Q_S$
and $\mathcal J_j=I_j\sqcup\{+,-\}$ on $Q_S$.

Direct calculation gives the following exchange identities.

\begin{lemma}\label{lem:commBB}
    Fix $i\notin Q_0^\tau$. If $ j\notin\{i,\tau i\} $ or $ r\neq s $, we have $  X_{j,s} X_{i,r}=C_{i,j,r,s} X_{i,r} X_{j,s} $, where
    \begin{equation*}
        C_{i,j,r,s} = C_{i,j}(x_{i,r},x_{j,s}), \qquad C_{i,j}(u,v):=\frac{q^{c_{ij}}u-v}{u-q^{c_{ij}}v}.
    \end{equation*}
\end{lemma}

\begin{lemma}\label{lem:commBT}
    For any $i$ and nonfixed $j$, we have $  X_{j,s}\hat{\Theta}_i(u)=D_{i,j,s}(u)\hat{\Theta}_i(u) X_{j,s} $, where
    \begin{equation*}
        D_{i,j,s}(u):=\frac{u-q^{c_{ij}-1}x_{j,s}^{-1}}{q^{c_{ij}}u-q^{-1}x_{j,s}^{-1}} \cdot \frac{q^{c_{\tau i,j}}u-qx_{j,s}}{u-q^{c_{\tau i,j}+1}x_{j,s}}=C_{i,j}(u,q^{-1}x_{j,s}^{-1})^{-1}C_{\tau i,j}(u,qx_{j,s}).
    \end{equation*}
\end{lemma}

\begin{lemma}\label{lem:valUnderlineT}
    For $i\notin Q_0^\tau$ and $1\leq r\leq v_i$, set
    \begin{equation*}
        \underline{\hat{\Theta}}_{i,r}(u) := \frac{(1-x_{i,r}/(q^{-1}u^{-1}))(1-x_{\tau i,r}/(q^{-1}u))}{(q-q^{-1})^2} \hat{\Theta}_i(u),
    \end{equation*}
    then for every $ k\in\mathbb{Z} $,
    \begin{equation*}
        \underline{\hat{\Theta}}_{i,r}(q^{-2k-1}x_{i,r}^{-1}) = q^{1-(k+1)c_{i,\tau i}} \left. \frac{u-q^{(2k+1)c_{i,\tau i}}u^{-1}}{u-u^{-1}} \right|_{u=q^{-2k-1}x_{i,r}^{-1}} \cdot D_{i,r}^k  X_{i,r} X_{\tau i,r} D_{\tau i,r}^k.
    \end{equation*}
\end{lemma}
\begin{proof}
    We multiply $\underline{\hat{\Theta}}_{i,r}(q^{-2k-1}x_{i,r}^{-1})$ by $1=D_{i,r}^{k+1}D_{\tau i,r}^{k+1}$, and move $D_{i,r}^{k+1}$ to the left to get the desired formula.
\end{proof}

For $k=-1$, both $D_{i,r}^{-1}  X_{i,r}$ and $X_{\tau i,r} D_{\tau i,r}^{-1}$ are scalars. Hence,
\[D_{i,r}^{-1}  X_{i,r} X_{\tau i,r} D_{\tau i,r}^{-1}=X_{\tau i,r} D_{\tau i,r}^{-1}D_{i,r}^{-1}  X_{i,r}=X_{\tau i,r}X_{i,r}.\]
The cases $k=0,-1$ of Lemma~\ref{lem:valUnderlineT} give the following residue formulas. 
\begin{lemma}\label{lem:resT0}
    If $c_{i,\tau i}=0$, the function $\hat{\Theta}_i(u) $ has only possible finite nonzero simple poles at $ q^{\pm 1}x_{i,r}^{-1} (1\leq r\leq v_i) $, and its residues are given by
    \begin{gather*}
        \Res_{u=q^{-1}x_{i,r}^{-1}}\frac{\hat{\Theta}_i(u)}{(q-q^{-1})u}du= X_{i,r} X_{\tau i,r},\quad
        \Res_{u=qx_{i,r}^{-1}}\frac{\hat{\Theta}_i(u)}{(q-q^{-1})u}du=- X_{\tau i,r} X_{i,r}.
    \end{gather*}
\end{lemma}

\begin{lemma}\label{lem:resT}
    For $c_{i,\tau i}=-1$, the function $ (u-u^{-1})\hat{\Theta}_i(u) $ has only possible finite nonzero simple poles at $ q^{\pm 1}x_{i,r}^{-1} (1\leq r\leq v_i) $, and its residues are given by
    \begin{gather*}
        \Res_{u=q^{-1}x_{i,r}^{-1}}\frac{(u-u^{-1})\hat{\Theta}_i(u)}{(q^2-1)u}du=(q^{-1}x_{i,r}^{-1}-x_{i,r}) X_{i,r} X_{\tau i,r}, \\
        \Res_{u=qx_{i,r}^{-1}}\frac{(u-u^{-1})\hat{\Theta}_i(u)}{(q^2-1)u}du=(q^{-1}x_{i,r}-x_{i,r}^{-1}) X_{\tau i,r} X_{i,r}.
    \end{gather*}
\end{lemma}

We now verify the relations in Definition~\ref{def:iquant}. Relation~\eqref{rel:HH} is obvious. We first consider the relations involving a nonfixed current $\hat{B}_i(u)$.
Relation~\eqref{rel:HB}, when $j$ is nonfixed, follows from Lemma~\ref{lem:commBT} and the identities
\begin{align*}
    \delta(qx_{j,s}v)D_{i,j,s}(u)^+ &= \delta(qx_{j,s}v)\left( C_{i,j}(u,v)^{-1}C_{\tau i,j}(u,v^{-1}) \right) ^+,\\
    C_{\tau i,j}(u^{-1},v) &= C_{\tau i,j}(u,v^{-1})^{-1}.
\end{align*}
Here the delta substitution is justified coefficient-wise by Lemma~\ref{lem:coefficientwise-delta}.
The relation~\eqref{rel:BB4} with at least one nonfixed vertex follows directly from Lemma~\ref{lem:commBB}.
For relation~\eqref{rel:BB5}, first suppose $j\neq i$ and $j\neq\tau i$. Lemma~\ref{lem:commBB} applies to every pair $(r,s)$, and the identity
\begin{equation*}
    (q^{-1}x_{j,s}^{-1}-q^{c_{ij}-1}x_{i,r}^{-1})C_{i,j}(x_{i,r},x_{j,s})=q^{c_{ij}-1}x_{j,s}^{-1}-q^{-1}x_{i,r}^{-1}
\end{equation*}
gives $(v-q^{c_{ij}}u)\hat{B}_j(v)\hat{B}_i(u)=(q^{c_{ij}}v-u)\hat{B}_i(u)\hat{B}_j(v)$. For nonfixed $i=j$, the same calculation applies to $r\neq s$. The diagonal terms vanish separately: after moving the first shift operator,
\[
 (u-q^2v)\delta(qx_{i,r}u)\delta(q^3x_{i,r}v)=0,
\]
and the reversed order is identical with $u,v$ exchanged. Thus
\eqref{rel:BB5} also holds for nonfixed $i=j$.

\subsection{Relation~\eqref{rel:BB1}}
For $c_{i,\tau i}=0$, we must prove
\begin{equation*}
    \hat{B}_i(u)\hat{B}_{\tau i}(v)-\hat{B}_{\tau i}(v)\hat{B}_i(u) = \frac{\delta(uv)}{q-q^{-1}}(\hat{\Theta}_i^+(u)-\hat{\Theta}_{\tau i}^+(v)).
\end{equation*}

By Lemma~\ref{lem:commBB}, for all $ r\neq s $, the terms $  X_{i,r} X_{\tau i,s} $ and $  X_{\tau i,s} X_{i,r} $ on the left-hand side cancel. Hence
\begin{align*}
    LHS 
    &= \sum_{r=1}^{v_i}\delta(qx_{i,r}u)\delta(uv) X_{i,r} X_{\tau i,r}-\sum_{r=1}^{v_i}\delta(q^{-1}x_{i,r}u)\delta(uv) X_{\tau i,r} X_{i,r}.
\end{align*}

As rational functions, $ \hat{\Theta}_i(u)=\hat{\Theta}_{\tau i}(u^{-1}) $. Their prescribed Laurent expansions are taken at $u=\infty$ and $u=0$, respectively.
By Lemma~\ref{lem:res},
\begin{align*}
    RHS 
    &= \delta(uv)\sum_k\delta(a_k^{-1}u)\Res_{u=a_k}\frac{\hat{\Theta}_i(u)}{(q-q^{-1})u}du.
\end{align*}
Applying Lemma~\ref{lem:resT0}, we get $ LHS=RHS $.

\subsection{Relation~\eqref{rel:BB3}}
For $c_{i,\tau i}=-1$, we must prove
\begin{equation*}
    (u-q^{-1}v)\hat{B}_i(u)\hat{B}_{\tau i}(v)+(v-q^{-1}u)\hat{B}_{\tau i}(v)\hat{B}_i(u)=\frac{\delta(uv)(u-v)}{q^2-1}(\hat{\Theta}_i^+(u)-\hat{\Theta}_{\tau i}^+(v)).
\end{equation*}

By the same reasoning, for all $ r\neq s $, the terms $  X_{i,r} X_{\tau i,s} $ and $  X_{\tau i,s} X_{i,r} $ on the left-hand side cancel. Hence
\begin{align*}
    LHS 
    &= \sum_{r=1}^{v_i}(q^{-1}x_{i,r}^{-1}-x_{i,r})\delta(qx_{i,r}u)\delta(uv) X_{i,r} X_{\tau i,r} \\
    &\mspace{24mu} +\sum_{r=1}^{v_i}(q^{-1}x_{i,r}-x_{i,r}^{-1})\delta(q^{-1}x_{i,r}u)\delta(uv) X_{\tau i,r} X_{i,r}.
\end{align*}

By Lemma~\ref{lem:res},
\begin{align*}
    RHS= \delta(uv)\sum_k\delta(a_k^{-1}u)\Res_{u=a_k}\frac{(u-u^{-1})\hat{\Theta}_i(u)}{(q^2-1)u}du.
\end{align*}
Applying Lemma~\ref{lem:resT}, we get $ LHS=RHS $.

\subsection{The Serre relation~\eqref{rel:Serre1}}\label{sec:serre1}

Assume $i\notin Q_0^\tau$, $j\neq\tau i$, and $c_{ij}=-1$. Write $\hat{\mathbb{S}}_{i,j}$ for the left-hand side of
\eqref{rel:Serre1} evaluated at $B_a=\hat{B}_a$. Set
\begin{equation*}
    P(u_1,u_2,v)=(u_1-qv)(u_2-qv),\qquad L_i(u_1,u_2)=\Sym_{u_1,u_2}\bigl( (u_1-q^2u_2)\hat{B}_i(u_1)\hat{B}_i(u_2) \bigr).
\end{equation*}

The quadratic relation~\eqref{rel:BB5} gives
\begin{equation*}
    (u-qv)\hat{B}_j(v)\hat{B}_i(u)=(qu-v)\hat{B}_i(u)\hat{B}_j(v).
\end{equation*}
Consequently,
\begin{align*}
    & P\cdot\hat{\mathbb{S}}_{i,j}(u_1,u_2\mid v) \\
    ={}& \Sym_{u_1,u_2}(P\cdot\hat{B}_i(u_1)\hat{B}_i(u_2)\hat{B}_j(v)-P\cdot[2]\hat{B}_i(u_1)\hat{B}_j(v)\hat{B}_i(u_2)+P\cdot\hat{B}_j(v)\hat{B}_i(u_1)\hat{B}_i(u_2)) \\
    ={}& \Sym_{u_1,u_2}\big((P-[2](u_1-qv)(qu_2-v)+(qu_1-v)(qu_2-v))\cdot\hat{B}_i(u_1)\hat{B}_i(u_2)\hat{B}_j(v)\big) \\
    ={}& \Sym_{u_1,u_2}\big((q^{-1}-q)(u_1-q^2u_2)v\cdot\hat{B}_i(u_1)\hat{B}_i(u_2)\hat{B}_j(v)\big) \\
    ={}& (q^{-1}-q)L_i(u_1,u_2)v\hat{B}_j(v).
\end{align*}
This calculation uses only the mixed quadratic relation and therefore also applies when $i$ is fixed. For nonfixed $i$, relation~\eqref{rel:BB5} with $j=i$ gives $L_i(u_1,u_2)=0$. Hence
\begin{equation*}
    P\cdot\hat{\mathbb{S}}_{i,j}(u_1,u_2\mid v)=0.
\end{equation*}

After normal ordering, $\hat{\mathbb{S}}_{i,j}(u_1,u_2\mid v)$ is a finite delta sum. If $j$ is nonfixed or orthogonal, since normal ordering preserves the odd $q$-exponents of its support values, $u_a\neq qv$ on every delta support. If $j$ is symplectic, physical summands are still handled by the same parity argument; for a scalar summand $v=\epsilon$, the factors $u_a-q\epsilon$ are nonzero by coordinate independence. Thus $P$ is nonzero on every support, and Lemma~\ref{lem:delta-cancellation} proves \eqref{rel:Serre1}.

\subsection{The Serre relation~\eqref{rel:Serre3}}

Assume $c_{i,\tau i}=-1$ and retain the notation $P$ and $L_i$ from Section~\ref{sec:serre1}. Let $R_{i,\tau i}^{(3)}$ denote the right-hand side of \eqref{rel:Serre3}, with the prescribed expansions. Set
\begin{equation*}
    E_i(u,v)=\frac{\delta(uv)(u-v)}{q^2-1}\bigl(\hat{\Theta}_i^+(u)-\hat{\Theta}_i^-(u)\bigr).
\end{equation*}
Since $\hat{\Theta}_i(u)=\hat{\Theta}_{\tau i}(u^{-1})$, relation~\eqref{rel:BB3} becomes
\begin{equation*}
    (u-qv)\hat{B}_{\tau i}(v)\hat{B}_i(u)=(qu-v)\hat{B}_i(u)\hat{B}_{\tau i}(v)-qE_i(u,v).
\end{equation*}
Reordering as in Section~\ref{sec:serre1} and using $L_i(u_1,u_2)=0$, we obtain
\begin{align*}
    P\cdot\hat{\mathbb{S}}_{i,\tau i}(u_1,u_2\mid v)&=\Sym_{u_1,u_2}\Bigl( 0+[2](u_1-qv)\hat{B}_i(u_1)qE_i(u_2,v) \\
    &\mspace{24mu}-(u_2-qv)qE_i(u_1,v)\hat{B}_i(u_2)-(qu_1-v)\hat{B}_i(u_1)qE_i(u_2,v)
    \Bigr) \\
    &=\Sym_{u_1,u_2}\Bigl( (u_1-q^3v)\hat{B}_i(u_1)E_i(u_2,v)-q(u_1-qv)E_i(u_2,v)\hat{B}_i(u_1)
    \Bigr) .
\end{align*}

Lemma~\ref{lem:commBT} gives
\begin{equation*}
    \hat{B}_i(u_1)\hat{\Theta}_i^\pm(u_2)=D_{u_1}^\pm(u_2)\hat{\Theta}_i^\pm(u_2)\hat{B}_i(u_1),
\end{equation*}
where
\begin{equation*}
    D_v(u)=\frac{u-q^2v}{q^2u-v}\frac{u^{-1}-q^{-1}v}{q^{-1}u^{-1}-v}.
\end{equation*}

A direct calculation gives the rational identity
\begin{equation*}
    \frac{u-u^{-1}}{q^2-1}
    \Bigl((w-q^3u^{-1})D_w(u)-q(w-qu^{-1})\Bigr)
    \hat{\Theta}_i(u)
    =P(w,u,u^{-1})\mathcal H_{i,w}^{(3)}(u).
\end{equation*}
Applying the two exchange identities and expanding this identity
in the respective directions yields
\begin{align*}
    P\cdot\hat{\mathbb{S}}_{i,\tau i}
    &=P\cdot \Sym_{u_1,u_2}\delta(u_2v)
    \Bigl(
        \mathcal H_{i,u_1}^{(3),+}(u_2)
        -\mathcal H_{i,u_1}^{(3),-}(u_2)
    \Bigr)\hat{B}_i(u_1)\\
    &=P\cdot R_{i,\tau i}^{(3)}.
\end{align*}

It remains to justify cancellation of $P$. The delta supports in left-hand side still have odd $q$-exponents. The delta supports in right hand side are of the form $\delta(qx_{i,r}u_1)\delta(a/u_2)\delta(av)$ or with $u_1,u_2$ switched, where $a$ is a nonzero finite pole of $\mathcal H_{i,q^{-1}x_{i,r}^{-1}}^{(3)}(u)$. Since $c_{i,\tau i}=-1$, the numerator contains a factor $u-x_{i,r}$ coming from either $\hat V_{\tau i}(u^{-1})$ or $\hat V_i(u)$. Thus $u=x_{i,r}$ is not a pole. By Lemma~\ref{lem:resT}, the remaining finite nonzero poles
are simple and are contained in
\begin{equation*}
    \{q^{-1}x_{i,s}^{-1},qx_{i,s}^{-1}\mid 1\le s\le v_i\}
    \cup\{q^{-3}x_{i,r}^{-1}\}.
\end{equation*}
They also have odd $q$-exponents. The parity argument of Section~\ref{sec:serre1} shows that $P$ is nonzero on every support of both sides. Lemma~\ref{lem:delta-cancellation} proves \eqref{rel:Serre3}.

\section{Orthogonal fixed vertices}\label{sec:appen2}

\subsection{Cartan and mixed exchange relations}
The Cartan currents commute because their coefficients are rational
functions. For a fixed vertex $j$ and a physical index $s$, direct
computation gives
\[
 X_{j,s}\hat{\Theta}_a(u)=
 \frac{u-q^{c_{aj}-1}x_{j,s}^{-1}}{q^{c_{aj}}u-q^{-1}x_{j,s}^{-1}}
 \frac{q^{c_{\tau a,j}}u-qx_{j,s}}
 {u-q^{c_{\tau a,j}+1}x_{j,s}}
 \hat{\Theta}_a(u)X_{j,s}.
\]
This is the same identity as in Lemma~\ref{lem:commBT}. Expanding at $u=\infty$ and applying Lemma~\ref{lem:coefficientwise-delta} to each current summand gives \eqref{rel:HB}. At a symplectic vertex, the coefficient
$X_{j,\epsilon}$ commutes with $\hat\Theta_a(u)$. The rational
factor in \eqref{rel:HB} is $1$ at $v=\epsilon$, because
$c_{aj}=c_{\tau a,j}$, so these terms satisfy the same relation.

For distinct vertices $a,j$ with $j\ne\tau a$, the summand exchange is
\[
 X_{j,s}X_{a,r}=C_{a,j}(x_{a,r},x_{j,s})X_{a,r}X_{j,s}
\]
whenever one of them is nonfixed. This includes the scalar indices of a
symplectic fixed vertex. Clearing the supportwise denominator proves
\eqref{rel:BB5}; for $c_{aj}=0$ it proves \eqref{rel:BB4}.
Two distinct fixed vertices are nonadjacent. Their coefficients may
share neighbor variables, but their shifts act on distinct fixed
factors, so their currents commute. Together with Appendix~\ref{sec:appen1},
this proves all Cartan and mixed quadratic relations.

The argument of Section~\ref{sec:serre1} also proves \eqref{rel:Serre1} for an orthogonal neighbor, since the mixed quadratic relation has now been verified and all delta support values still have odd $q$-exponents.

\subsection{The fixed-vertex quadratic relation~\eqref{rel:BB2}}

Let $r'=v_i+1-r$ and $\mathcal{I}_i=\{1\leq r\leq v_i\mid r\ne r'\}$. Write $x_r:=x_{i,r}$ and $D_r:=D_{i,r}$; the central weight $x_{i,(v_i+1)/2}=1$ is excluded when $v_i$ is odd. As in Appendix \ref{sec:appen1}, denote the coefficient of each delta summand by $ X_r$, and set
\begin{equation*}
    A(z):=W_i(z)\prod_{\substack{h\in Q_1\\s(h)=i}}\hat{V}_{t(h)}(qz),\qquad
    X_r:=\frac{(1-x_r^{-2})A(x_r)}{V_{i,r}(x_r)}D_r\quad(r\in\mathcal{I}_i).
\end{equation*}
Since $i\in Q_0^\tau$, we have $M_i=1$.
Moreover, $x_{r'}=x_r^{-1}$, so the $r'$-factor in
$V_{i,r}(x_r)$ is $1-x_r^{-2}$. Cancelling this factor gives
\[
V_i(z)=\prod_{s=1}^{v_i}(1-x_s/z),
\qquad
X_r=
\frac{A(x_r)}
     {\prod_{\substack{1\leq s\leq v_i\\s\ne r,r'}}
      (1-x_s/x_r)}
D_r
\quad(r\in\mathcal I_i).
\]
The fixed-vertex series in the orthogonal case are
\begin{align}
    \hat{B}_i(u) &= \sum_{r\in\mathcal{I}_i}\delta(qx_ru) X_r, \\
    \hat{\Theta}_i(u) &= -q(q-q^{-1})^2(1-q^2u^2)(1-q^2u^{-2}) \frac{A(q^{-1}u^{-1})A(q^{-1}u)}{V_i(q^{-1}u^{-1})V_i(q^{-1}u)}. \label{eq:Thetan-O-short}
\end{align}
In particular, $\hat{\Theta}_i(u)=\hat{\Theta}_i(u^{-1})$ as a rational function. For $s\notin\{r,r'\}$, moving the difference operators past the coefficients gives
\begin{equation*}
    X_s X_r=\frac{q^2x_r-x_s}{x_r-q^2x_s} X_r X_s=C_{i,i}(x_r,x_s) X_r X_s.
\end{equation*}

For $s\notin\{r,r'\}$, the coefficient of $\delta(qx_ru)\delta(qx_sv) X_r X_s$ in the LHS of \eqref{rel:BB2} is
\begin{equation*}
    (q^{-1}x_r^{-1}-qx_s^{-1})+(q^{-1}x_s^{-1}-qx_r^{-1})\frac{q^2x_r-x_s}{x_r-q^2x_s}=0.
\end{equation*}
For $r=s$, the corresponding term vanishes because
\begin{equation*}
    (u-q^2v)\delta(qx_ru)\delta(q^3x_rv)=0.
\end{equation*}
For opposite indices, moving the shift through the delta function gives
\begin{equation*}
    X_r\delta(qx_{r'}v)=\delta(q^{-1}x_r^{-1}v) X_r.
\end{equation*}
These terms are supported on $uv=1$. Thus
\begin{align}
    LHS &= \sum_{r\in\mathcal{I}_i}\Big[ (u-q^2v)\delta(qx_ru)\delta(q^{-1}x_r^{-1}v) + (v-q^2u)\delta(qx_rv)\delta(q^{-1}x_r^{-1}u) \Big]  X_r X_{r'} \notag\\
    &= \delta(uv)\sum_{r\in\mathcal{I}_i}(q^{-1}x_r^{-1}-q^3x_r)\bigl( \delta(qx_ru)+\delta(q^{-1}x_r^{-1}u) \bigr)  X_r X_{r'}. \label{eq:BBn-O-LHS}
\end{align}

Set
\begin{equation*}
    \Gamma(u) := \frac{u-u^{-1}}{q^{-1}-q}\hat{\Theta}_i(u).
\end{equation*}
Its only possible finite nonzero poles are the simple poles $q^{-1}x_r^{-1},qx_r$ for $r\in\mathcal{I}_i$. In odd dimension, the factors from the middle weight $x_{(v_i+1)/2}=1$ in each denominator of \eqref{eq:Thetan-O-short} are cancelled by $(1-q^2u^2)(1-q^2u^{-2})$; they introduce no extra pole. Substitution at $u=q^{-1}x_r^{-1}$ gives
\begin{equation*}
    \Res_{u=q^{-1}x_r^{-1}}\frac{\Gamma(u)}{u}du = \frac{1-q^4x_r^2}{qx_r}\frac{A(x_r)A(q^{-2}x_r^{-1})}{\displaystyle\prod_{s\ne r,r'}(1-x_s/x_r)(1-q^2x_sx_r)} = (q^{-1}x_r^{-1}-q^3x_r) X_r X_{r'}.
\end{equation*}
The residue at $qx_r$ is the same, by the inversion symmetry of $\hat{\Theta}_i$. Hence Lemma \ref{lem:res} and \eqref{eq:BBn-O-LHS} yield
\begin{equation*}
    LHS = \delta(uv)\bigl( \Gamma(u)^+ - \Gamma(u)^- \bigr) = \frac{\delta(uv)(u-v)}{q^{-1}-q}\bigl( \hat{\Theta}_i^+(u)-\hat{\Theta}_i^+(v) \bigr) .
\end{equation*}
This is \eqref{rel:BB2}.

\subsection{The fixed-vertex Serre relation~\eqref{rel:Serre2}}

Let $j$ satisfy $c_{ij}=-1$. Since the fixed vertices are pairwise nonadjacent, $j$ is nonfixed. Write $\hat{\mathbb{S}}_{i,j}$ for the left-hand side of \eqref{rel:Serre2} evaluated at $B_a=\hat{B}_a$.

Retain the notation $P$, $L_i$, and $\hat{\mathbb{S}}_{i,j}$ from Section~\ref{sec:serre1}, and let $R_{i,j}^{(2)}$ denote the right-hand side of \eqref{rel:Serre2}, with the prescribed expansions.

Write $y_s=x_{j,s}$, $Y_s=X_{j,s}$, $a_s=q^{-1}y_s^{-1}$.
By Lemma~\ref{lem:coefficientwise-delta} and Remark~\ref{rem:expansions},
\begin{equation*}
    R_{i,j}^{(2)}=\sum_s\delta(qy_sv)\delta(u_1u_2)\bigl(\mathcal H_{i,a_s}^{(2),+}(u_1)-\mathcal H_{i,a_s}^{(2),-}(u_1)\bigr)Y_s.
\end{equation*}

Since
\begin{equation*}
    P(u,u^{-1},q^{-1}y_s^{-1})\mathcal H_{i,a_s}^{(2)}(u)=(u-u^{-1})q^{-1}y_s^{-1}\hat{\Theta}_i(u),
\end{equation*}
the fixed-vertex quadratic relation \eqref{rel:BB2} and the calculation in Section~\ref{sec:serre1} give
\begin{align*}
    P\cdot R_{i,j}^{(2)}&=\sum_s\delta(qy_sv)\delta(u_1u_2)\bigl(((u_1-u_1^{-1})\hat{\Theta}_i(u_1))^+-((u_1-u_1^{-1})\hat{\Theta}_i(u_1))^-\bigr)q^{-1}y_s^{-1}Y_s \\
    &=\delta(u_1u_2)(u_1-u_2)\bigl(\hat{\Theta}_i^+(u_1)-\hat{\Theta}_i^-(u_1)\bigr)v\sum_s\delta(qy_sv)Y_s \\
    &=(q^{-1}-q)L_i(u_1,u_2)v\hat{B}_j(v) \\
    &=P\cdot\hat{\mathbb{S}}_{i,j}.
\end{align*}

Since $c_{ij}=-1$, either $i\to j$ or $j\to i$; the latter is equivalent to $i\to\tau j$. In either case, the factors $(u-y_s^{-1})(u^{-1}-y_s^{-1})$ in the denominator of $h_{a_{s}}^{(2)}(u)$ cancel with numerator factors in $A(q^{-1}u^{-1})A(q^{-1}u)$. Thus $\mathcal H_{i,a_s}^{(2),+}-\mathcal H_{i,a_s}^{(2),-}$ has no support points beyond those of $\Gamma^+-\Gamma^-$. The parity argument in Section~\ref{sec:serre1} therefore shows that $P$ is nonzero on every delta support of both sides. Lemma~\ref{lem:delta-cancellation} proves \eqref{rel:Serre2}.

\section{Symplectic fixed vertices}\label{sec:appen3}
Fix $i\in Q_S$. The common Cartan and mixed relations were checked
in Appendix~\ref{sec:appen2}, and the ordinary Serre relation with
nonfixed repeated vertex in Appendix~\ref{sec:appen1}. We verify
\eqref{rel:BB2} and \eqref{rel:Serre2}, keeping both scalar indices
of \eqref{eq:Bsymplectic} throughout.

\subsection{The fixed-vertex quadratic relation \eqref{rel:BB2}}
We need to check
\begin{equation}\label{eq:BBnC}
    (u-q^2v)\hat{B}_i(u)\hat{B}_i(v)+(v-q^2u)\hat{B}_i(v)\hat{B}_i(u)
    =\frac{\delta(uv)(u-v)}{q^{-1}-q}
    \big(\hat{\Theta}_i^+(u)-\hat{\Theta}_i^+(v)\big).
\end{equation}

Recall $\mathcal{J}_i:=\{1,2,\dots,v_i\}\sqcup\{+,-\}$. For $1\leq r\leq v_i$, write $x_r:=x_{i,r}$ and $D_r:=D_{i,r}$; for $\epsilon\in\{+,-\}$, put $x_\epsilon:=\epsilon q^{-1}$. Let
\begin{equation*}
    A(z):=W_i(z)\prod_{\substack{h\in Q_1\\s(h)=i}}\hat{V}_{t(h)}(qz).
\end{equation*}
Recall $r':=1+v_i-r$, $x_{r'}=x_r^{-1}$ and $D_{r'}=D_r^{-1}$ for $1\leq r\leq v_i$.

Define
\begin{equation*}
    X_r:=\frac{A(x_r)}{(1-q^{-2}x_r^{-2})V_{i,r}(x_r)}D_r,
    \qquad
    X_\epsilon:=\frac{1}{2}\frac{A(\epsilon q^{-1})}
    {V_i(\epsilon q^{-1})},\quad\epsilon\in\{+,-\}.
\end{equation*}
Here $+$ and $-$ are the two extra index symbols; in scalar expressions they have values $+1$ and $-1$, respectively. Then
\begin{align*}
    \hat{B}_i(u)&=\sum_{r\in \mathcal{J}_i}\delta(qx_ru)X_r, \\
    \hat\Theta_i(u)&=\frac{q(q-q^{-1})^2}{(u-u^{-1})^2} \frac{A(q^{-1}u^{-1})A(q^{-1}u)}{V_i(q^{-1}u^{-1})V_i(q^{-1}u)}.
\end{align*}

By a direct check, we get that for $s\notin\{r,r'\}$, 
\begin{equation}
    X_sX_r=\frac{q^2x_r-x_s}{x_r-q^2x_s}X_rX_s.
\end{equation}
The same identity holds if $s$ is replaced by $\epsilon\in\{+,-\}$.  It also holds for the pair
$+,-$.  

The cancellation of distinct, non-opposite physical indices and the vanishing of repeated physical indices are the calculations preceding \eqref{eq:BBn-O-LHS}. The mixed terms involving $X_r$ and $X_\epsilon$, as well as the terms involving $X_+$ and $X_-$, cancel in the same way.
Only the pairs $r,r'$ and the squares of $X_\pm$ remain.
Consequently, the left-hand side of \eqref{eq:BBnC} is
\begin{align}
 &\sum_{\epsilon\in\{+,-\}}2\epsilon(1-q^2)
   \delta(\epsilon u)\delta(\epsilon v)X_\epsilon^2\notag\\
 &\quad+\sum_{r=1}^{v_i}\Big(
 (u-q^2v)\delta(qx_ru)\delta(q^{-1}x_r^{-1}v)
 +(v-q^2u)\delta(qx_rv)\delta(q^{-1}x_r^{-1}u)
 \Big)X_rX_{r'}.
 \label{eq:BBnC-LHS}
\end{align}

It remains to identify \eqref{eq:BBnC-LHS} with the right-hand side.
Let
\[
 \Gamma(u):=\frac{u-u^{-1}}{q^{-1}-q}\hat\Theta_i(u).
\]
As a rational function, $\hat\Theta_i(u)=\hat\Theta_i(u^{-1})$.
The finite nonzero poles of $\Gamma(u)$ are simple and are contained in
\[
 \{1,-1\}\cup
 \{q^{-1}x_r^{-1},qx_r\mid 1\leq r\leq v_i\}.
\]
A direct computation gives
\begin{align*}
 \Res_{u=\epsilon}\frac{\Gamma(u)}u du
 &=2\epsilon(1-q^2)X_\epsilon^2,
 \qquad \epsilon\in\{+,-\},\\
 \Res_{u=q^{-1}x_r^{-1}}\frac{\Gamma(u)}u du
 &=\Res_{u=qx_r}\frac{\Gamma(u)}u\,du=(q^{-1}x_r^{-1}-q^3x_r)X_rX_{r'}.
\end{align*}
Applying the formal-distribution identity of Lemma \ref{lem:res},
we obtain
\begin{align*}
 \delta(uv)&\bigl(\Gamma(u)^+-\Gamma(u)^-\bigr)
 =\sum_{\epsilon\in\{+,-\}}2\epsilon(1-q^2)
 \delta(\epsilon u)\delta(\epsilon v)X_\epsilon^2 \notag\\
 &+\sum_{r=1}^{v_i}\Big[
 (u-q^2v)\delta(qx_ru)\delta(q^{-1}x_r^{-1}v)+(v-q^2u)\delta(qx_rv)
 \delta(q^{-1}x_r^{-1}u)\Big]X_rX_{r'}.
\end{align*}
The right-hand side is exactly \eqref{eq:BBnC-LHS}.
On the other hand, since
$\hat\Theta_i(u)=\hat\Theta_i(u^{-1})$,
\begin{equation*}
    \delta(uv)\bigl(\Gamma(u)^+-\Gamma(u)^-\bigr)
 =\frac{\delta(uv)(u-v)}{q^{-1}-q}
 \bigl(\hat\Theta_i^+(u)-\hat\Theta_i^+(v)\bigr).
\end{equation*}
This proves \eqref{eq:BBnC}.

\subsection{The fixed-vertex Serre relation \eqref{rel:Serre2}}

Let $j$ satisfy $c_{ij}=-1$; it is nonfixed. Use the notation of the corresponding proof in Appendix~\ref{sec:appen2}. The mixed quadratic relation, including both scalar summands of $\hat B_i$, and \eqref{rel:BB2} give the same identity
\begin{equation*}
    P\cdot\hat{\mathbb{S}}_{i,j}=P\cdot R_{i,j}^{(2)}.
\end{equation*}

The cancellation argument in Appendix~\ref{sec:appen2} also applies: the additional supports $\pm1$ are covered by coordinate independence. Thus \eqref{rel:Serre2} follows.

\bibliographystyle{alpha}
\bibliography{iquantumCoulomb.bib}

@article{BFN18,
    AUTHOR = {Braverman, A. and Finkelberg, M. and Nakajima,
              H.},
     TITLE = {Towards a mathematical definition of {C}oulomb branches of
              3-dimensional {$\mathcal{N}=4$} gauge theories, {II}},
   JOURNAL = {Adv. Theor. Math. Phys.},
  FJOURNAL = {Advances in Theoretical and Mathematical Physics},
    VOLUME = {22},
      YEAR = {2018},
    NUMBER = {5},
     PAGES = {1071--1147},
      ISSN = {1095-0761,1095-0753},
   MRCLASS = {57R57 (14J33 14N35 16G20 17B67 81T13)},
  MRNUMBER = {3952347},
MRREVIEWER = {Dave\ Auckly},
       DOI = {10.4310/ATMP.2018.v22.n5.a1},
       URL = {https://doi.org/10.4310/ATMP.2018.v22.n5.a1},
}

@article{BFN19,
    AUTHOR = {Braverman, A. and Finkelberg, M. and Nakajima,
              H.},
     TITLE = {Coulomb branches of {$3d$} {$\mathcal{N}=4$} quiver gauge theories
              and slices in the affine {G}rassmannian},
      NOTE = {With two appendices by Braverman, Finkelberg, Joel Kamnitzer,
              Ryosuke Kodera, Nakajima, Ben Webster and Alex Weekes},
   JOURNAL = {Adv. Theor. Math. Phys.},
  FJOURNAL = {Advances in Theoretical and Mathematical Physics},
    VOLUME = {23},
      YEAR = {2019},
    NUMBER = {1},
     PAGES = {75--166},
      ISSN = {1095-0761,1095-0753},
   MRCLASS = {57R57 (14M15 16G20 17B81 81T13)},
  MRNUMBER = {4020310},
       DOI = {10.4310/ATMP.2019.v23.n1.a3},
       URL = {https://doi.org/10.4310/ATMP.2019.v23.n1.a3},
}

@incollection {GKLOquant,
    AUTHOR = {Gerasimov, A. and Kharchev, S. and Lebedev, D. and
              Oblezin, S.},
     TITLE = {On a class of representations of quantum groups},
 BOOKTITLE = {Noncommutative geometry and representation theory in
              mathematical physics},
    SERIES = {Contemp. Math.},
    VOLUME = {391},
     PAGES = {101--110},
 PUBLISHER = {Amer. Math. Soc., Providence, RI},
      YEAR = {2005},
      ISBN = {0-8218-3718-4},
   MRCLASS = {17B37 (20G42)},
  MRNUMBER = {2184015},
MRREVIEWER = {Fabio\ Gavarini},
       DOI = {10.1090/conm/391/07321},
       URL = {https://doi.org/10.1090/conm/391/07321},
}

@article{bartlett2025gklo,
  title={{GKLO} representations of twisted {Y}angians in type {AI} and quantizations of symmetric quotients of the affine {G}rassmannian},
  author={Bartlett, R. and Przezdziecki, T. and Tappeiner, L.},
  year={2025},
    JOURNAL = {preprint},
    note    = {\href{https://arxiv.org/abs/2510.12706}{arXiv:2510.12706}},
    URL = {https://arxiv.org/abs/2510.12706},
    EPRINT = {2510.12706},
}

@article {dong2024flag,
    AUTHOR = {Dong, Z. and Ma, H.},
     TITLE = {Flag varieties of type {B}/{C} and twisted {Y}angian},
   JOURNAL = {Algebra Colloq.},
  FJOURNAL = {Algebra Colloquium},
    VOLUME = {31},
      YEAR = {2024},
    NUMBER = {1},
     PAGES = {129--148},
      ISSN = {1005-3867,0219-1733},
   MRCLASS = {17B37},
  MRNUMBER = {4717572},
MRREVIEWER = {Aleksandr\ Panov},
       DOI = {10.1142/s1005386724000129},
       URL = {https://doi.org/10.1142/s1005386724000129},
}

@article {FMX22,
    AUTHOR = {Fan, Z. and Ma, H. and Xiao, H.},
     TITLE = {Equivariant {K}-theory approach to {$\imath$}-quantum groups},
   JOURNAL = {Publ. Res. Inst. Math. Sci.},
  FJOURNAL = {Publications of the Research Institute for Mathematical
              Sciences},
    VOLUME = {58},
      YEAR = {2022},
    NUMBER = {3},
     PAGES = {635--668},
      ISSN = {0034-5318,1663-4926},
   MRCLASS = {20G42 (16T20 19L47)},
  MRNUMBER = {4458549},
MRREVIEWER = {Eusebio\ Gardella},
       DOI = {10.4171/prims/58-3-6},
       URL = {https://doi.org/10.4171/prims/58-3-6},
}

@incollection {FT19,
    AUTHOR = {Finkelberg, M. and Tsymbaliuk, A.},
     TITLE = {Multiplicative slices, relativistic {T}oda and shifted quantum
              affine algebras},
 BOOKTITLE = {Representations and nilpotent orbits of {L}ie algebraic
              systems},
    SERIES = {Progr. Math.},
    VOLUME = {330},
     PAGES = {133--304},
 PUBLISHER = {Birkh\"auser/Springer, Cham},
      YEAR = {2019},
      ISBN = {978-3-030-23530-7; 978-3-030-23531-4},
   MRCLASS = {17B37 (81R10 81T13)},
  MRNUMBER = {3971731},
MRREVIEWER = {Kyungyong\ Lee},
       DOI = {10.1007/978-3-030-23531-4\_6},
       URL = {https://doi.org/10.1007/978-3-030-23531-4_6},
}

@article {GKLO05,
    AUTHOR = {Gerasimov, A. and Kharchev, S. and Lebedev, D. and Oblezin,
              S.},
     TITLE = {On a class of representations of the {Y}angian and moduli
              space of monopoles},
   JOURNAL = {Comm. Math. Phys.},
  FJOURNAL = {Communications in Mathematical Physics},
    VOLUME = {260},
      YEAR = {2005},
    NUMBER = {3},
     PAGES = {511--525},
      ISSN = {0010-3616,1432-0916},
   MRCLASS = {53C07 (17B37 53D17 53D30)},
  MRNUMBER = {2182434},
MRREVIEWER = {Olivier\ G.\ Schiffmann},
       DOI = {10.1007/s00220-005-1417-3},
       URL = {https://doi.org/10.1007/s00220-005-1417-3},
}

@article {KWWY14,
    AUTHOR = {Kamnitzer, J. and Webster, B. and Weekes, A. and Yacobi,
              O.},
     TITLE = {Yangians and quantizations of slices in the affine
              {G}rassmannian},
   JOURNAL = {Algebra Number Theory},
  FJOURNAL = {Algebra \& Number Theory},
    VOLUME = {8},
      YEAR = {2014},
    NUMBER = {4},
     PAGES = {857--893},
      ISSN = {1937-0652,1944-7833},
   MRCLASS = {17B37 (14D24 14M15 20G15 53D55)},
  MRNUMBER = {3248988},
MRREVIEWER = {Christian\ Ohn},
       DOI = {10.2140/ant.2014.8.857},
       URL = {https://doi.org/10.2140/ant.2014.8.857},
}

@article{LWW,
    author = {Lu, K. and Wang, W. and Weekes, A.},
    title = {Shifted Twisted {Y}angians and affine {G}rassmannian islices},
    year={2025},
JOURNAL = {preprint},
note    = {\href{https://arxiv.org/abs/2510.10652}{arXiv:2510.10652}},
URL = {https://arxiv.org/abs/2510.10652},
    EPRINT = {2510.10652},
}

@article{lu2025shifted,
  title={Shifted twisted {Y}angians of quasi-split {ADE} types},
  author={Lu, K. and Wang, W. and Weekes, A.},
  year={2025},
  JOURNAL = {preprint}, note    = {\href{https://arxiv.org/abs/2512.19998}{arXiv:2512.19998}}, 
  URL = {https://arxiv.org/abs/2512.19998}, 
  EPRINT = {2512.19998},
}

@article{li2026gklorepresentationsshiftedquantum,
      title={{GKLO} representations for shifted quantum affine symmetric pairs},
  author={Li, J. and Przezdziecki, T.},
      year={2026},
    JOURNAL = {preprint},
    note    = {\href{https://arxiv.org/abs/2603.07250}{arXiv:2603.07250}},
    URL = {https://arxiv.org/abs/2603.07250},
    EPRINT = {2603.07250},
}

@article {NakCou1,
    AUTHOR = {Nakajima, H.},
     TITLE = {Towards a mathematical definition of {C}oulomb branches of
              3-dimensional {$\mathcal N=4$} gauge theories, {I}},
   JOURNAL = {Adv. Theor. Math. Phys.},
  FJOURNAL = {Advances in Theoretical and Mathematical Physics},
    VOLUME = {20},
      YEAR = {2016},
    NUMBER = {3},
     PAGES = {595--669},
      ISSN = {1095-0761,1095-0753},
   MRCLASS = {57R57 (14J33 14N35 16G20 17B67 81T13)},
  MRNUMBER = {3565863},
       DOI = {10.4310/ATMP.2016.v20.n3.a4},
       URL = {https://doi.org/10.4310/ATMP.2016.v20.n3.a4},
}

@article {SSX26,
    AUTHOR = {Shen, Y. and Su, C. and Xiong, R.},
     TITLE = {Quivers with involutions and shifted twisted {Y}angians via {C}oulomb Branches},
      Journal = {Comm. Math. Phys. (to appear)},
ARCHIVEPREFIX = {arXiv},
note    = {\href{https://arxiv.org/abs/2510.12118}{arXiv:2510.12118}},
URL = {https://arxiv.org/abs/2510.12118},
    EPRINT = {2510.12118},
    year={2026},
}

@article{SSX26b,
  title={Quantized {C}oulomb Branches of Separated Cotangent Type and Orthosymplectic Quivers},
  author={Shen, Y. and Su, C. and Xiong, R.},
  journal={arXiv preprint arXiv:2608.16091},
  year={2026}
}

@article{wang2025quivers,
  title={Quivers with involutions and shifted twisted {Y}angians via {C}oulomb branches {II}},
  author={Wang, Z.},
  year={2026},
  JOURNAL = {preprint}, note    = {\href{https://arxiv.org/abs/2601.00039}{arXiv:2601.00039}}, 
  URL = {https://arxiv.org/abs/2601.00039}, 
  EPRINT = {2601.00039},
}

@article{LWW26,
  title={Shifted affine iquantum groups of quasi-split {ADE} types},
  author={Lu, K. and Wang, W. and Weekes, A.},
  year={2026},
  JOURNAL = {preprint},
note    = {\href{https://arxiv.org/abs/2603.28446}{arXiv:2603.28446}},
URL = {https://arxiv.org/abs/2603.28446},
    EPRINT = {2603.28446},
}

@article{GV93,
    AUTHOR = {Ginzburg, V. and Vasserot, E.},
     TITLE = {Langlands reciprocity for affine quantum groups of type
              {$A_n$}},
   JOURNAL = {Internat. Math. Res. Notices},
  FJOURNAL = {International Mathematics Research Notices},
      YEAR = {1993},
    NUMBER = {3},
     PAGES = {67--85},
      ISSN = {1073-7928,1687-0247},
   MRCLASS = {17B37 (17B67)},
  MRNUMBER = {1208827},
MRREVIEWER = {Jie\ Du},
       DOI = {10.1155/S1073792893000078},
       URL = {https://doi.org/10.1155/S1073792893000078},
}

@article{Vas98,
  author={Vasserot, E.},
  title={Affine quantum groups and equivariant {$K$}-theory},
  journal={Transform. Groups},
  volume={3},
  number={3},
  year={1998},
  pages={269--299},
  doi={10.1007/BF01236876}
}

@article{Nak01,
  author={Nakajima, H.},
  title={Quiver varieties and finite-dimensional representations of quantum affine algebras},
  journal={J. Amer. Math. Soc.},
  volume={14},
  number={1},
  year={2001},
  pages={145--238},
  doi={10.1090/S0894-0347-00-00353-2}
}

@article{Var00,
  author={Varagnolo, M.},
  title={Quiver varieties and {Y}angians},
  journal={Lett. Math. Phys.},
  volume={53},
  number={4},
  year={2000},
  pages={273--283}
}

@article{Nak25,
  author={Nakajima, H.},
  title={Instantons on {ALE} spaces for classical groups, involutions on quiver varieties, and quantum symmetric pairs},
  journal={preprint},
  year={2025},
  note={\href{https://arxiv.org/abs/2510.13007}{arXiv:2510.13007}},
  eprint={2510.13007}
}

@article{SW26,
  author={Su, C. and Wang, W.},
  title={Affine {$\mathrm{i}$}quantum groups and {S}teinberg varieties of type {C}},
  journal={Selecta Math. (N.S.)},
  volume={32},
  year={2026},
  pages={50},
  doi={10.1007/s00029-026-01161-x}
}

@article{LSX26,
  author={Luo, L. and Su, C. and Xu, Z.},
  title={Affine {$\mathrm{i}$}quantum groups and {S}teinberg varieties of type {C}, {II}},
  journal={preprint},
  year={2026},
  note={\href{https://arxiv.org/abs/2605.09589}{arXiv:2605.09589}},
  eprint={2605.09589}
}

@article{SY26,
  author={Su, C. and Yang, Y.},
  title={Twisted {Y}angians and {S}teinberg varieties of type {C}},
  journal={preprint},
  year={2026},
  note={\href{https://arxiv.org/abs/2603.28084}{arXiv:2603.28084}},
  eprint={2603.28084}
}

@article {NP01,
    AUTHOR = {Ng\^o, B. C. and Polo, P.},
     TITLE = {R\'esolutions de {D}emazure affines et formule de
              {C}asselman-{S}halika g\'eom\'etrique},
   JOURNAL = {J. Algebraic Geom.},
  FJOURNAL = {Journal of Algebraic Geometry},
    VOLUME = {10},
      YEAR = {2001},
    NUMBER = {3},
     PAGES = {515--547},
      ISSN = {1056-3911,1534-7486},
   MRCLASS = {14F43 (14F17)},
  MRNUMBER = {1832331},
MRREVIEWER = {Nguy\cftil en Qu\^oc Th\'ang},
}

@incollection {Zhu17,
    AUTHOR = {Zhu, X.},
     TITLE = {An introduction to affine {G}rassmannians and the geometric
              {S}atake equivalence},
 BOOKTITLE = {Geometry of moduli spaces and representation theory},
    SERIES = {IAS/Park City Math. Ser.},
    VOLUME = {24},
     PAGES = {59--154},
 PUBLISHER = {Amer. Math. Soc., Providence, RI},
      YEAR = {2017},
      ISBN = {978-1-4704-3574-5},
   MRCLASS = {14M15 (14D24 20F65 22E57)},
  MRNUMBER = {3752460},
MRREVIEWER = {Felipe\ Zald\'ivar},
}

\end{document}